\documentclass[12pt]{amsart}
\usepackage{amscd,amssymb,latexsym}
\usepackage{fullpage}
\newcommand{\Mdef}[2]{\newcommand{#1}{\relax \ifmmode #2 \else $#2$\fi}}

\newcommand{\supp}{\mathrm{supp}}

\newcommand{\im}{\mathrm{im}}
\newcommand{\thick}{\mathrm{thick}}

\newcommand{\sm }{\wedge}

\newcommand{\tensor}{\otimes}

\newcommand{\Hom}{\mathrm{Hom}}

\newcommand{\Ext}{\mathrm{Ext}}
\Mdef{\bhom}{\mathbf{\hat{H}om}}

\Mdef{\Mod}{\mathrm{mod}}

\newcommand{\st}{\; | \;}

\newtheorem{thm}{Theorem}[section]
\newtheorem{lemma}[thm]{Lemma}
\newtheorem{prop}[thm]{Proposition}
\newtheorem{cor}[thm]{Corollary}

\theoremstyle{definition}

\newtheorem{defn}[thm]{Definition}
\newtheorem{cond}[thm]{Condition}

\newtheorem{example}[thm]{Example}

\newtheorem{remark}[thm]{Remark}

\newcommand{\qqed}{\qed \\[1ex]}
\renewenvironment{proof}[1][\hspace*{-.8ex}]{\noindent {\bf Proof #1:\;}}{\qqed}

\Mdef{\PH} {\Phi^H}
\Mdef{\PK} {\Phi^K}
\Mdef{\PL} {\Phi^L}
\Mdef{\PT} {\Phi^{\T}}

\Mdef{\ef}{E{\cF}_+}
\Mdef{\etf}{\widetilde{E}{\cF}}
\Mdef{\eg}{E{G}_+}
\Mdef{\etg}{\tilde{E}{G}}
\newcommand{\piA}{\pi^{\cA}}

\Mdef{\infl}{\mathrm{inf}}
\Mdef{\defl}{\mathrm{def}}
\Mdef{\res}{\mathrm{res}}
\Mdef{\ind}{\mathrm{ind}}
\Mdef{\coind}{\mathrm{coind}}

\Mdef{\univ}{\mathcal{U}}

\Mdef{\Fp}{\mathbb{F}_p}
\Mdef{\Zpinfty}{\Z /p^{\infty}}
\Mdef{\Zpadic}{\Z_p^{\wedge}}

\newcommand{\bi}{\begin{itemize}}
\newcommand{\be}{\begin{enumerate}}
\newcommand{\bc}{\begin{center}}
\newcommand{\bd}{\begin{description}}
\newcommand{\ei}{\end{itemize}}
\newcommand{\ee}{\end{enumerate}}
\newcommand{\ec}{\end{center}}
\newcommand{\ed}{\end{description}}

\newcommand{\lra}{\longrightarrow}

\newcommand{\iso}{\cong}

\newcommand{\Gspectra}{\mbox{$G$-{\bf spectra}}}

\newcommand{\spec}{\mathrm{spec}}

\newcommand{\rings}{\mathbf{Rings}}

\Mdef{\we}{\mathbf{we}}
\Mdef{\fib}{\mathbf{fib}}
\Mdef{\cof}{\mathbf{cof}}
\Mdef{\BI}{\mathcal{BI}}

\newcommand{\cofibre}{\mathrm{cofibre}}

\newcommand{\colim}{\mathop{  \mathop{\mathrm {lim}} \limits_\rightarrow} \nolimits}

\newcommand{\hocolim}{\mathop{  \mathop{\mathrm {holim}}\limits_\rightarrow} \nolimits}

\Mdef{\B}{\mathbb{B}}
\Mdef{\C}{\mathbb{C}}
\Mdef{\D}{\mathbb{D}}
\Mdef{\E}{\mathbb{E}}
\Mdef{\T}{\mathbb{T}}
\Mdef{\F}{\mathbb{F}}
\Mdef{\G}{\mathbb{G}}
\Mdef{\I}{\mathbb{I}}
\Mdef{\N}{\mathbb{N}}
\Mdef{\Q}{\mathbb{Q}}
\Mdef{\R}{\mathbb{R}}
\Mdef{\bbS}{\mathbb{S}}
\Mdef{\Z}{\mathbb{Z}}

\Mdef{\bA}{\mathbb{A}}
\Mdef{\bB}{\mathbb{B}}
\Mdef{\bC}{\mathbb{C}}
\Mdef{\bD}{\mathbb{D}}
\Mdef{\bE}{\mathbb{E}}
\Mdef{\bF}{\mathbb{F}}
\Mdef{\bG}{\mathbb{G}}
\Mdef{\bH}{\mathbb{H}}
\Mdef{\bI}{\mathbb{I}}
\Mdef{\bJ}{\mathbb{J}}
\Mdef{\bK}{\mathbb{K}}
\Mdef{\bL}{\mathbb{L}}
\Mdef{\bM}{\mathbb{M}}
\Mdef{\bN}{\mathbb{N}}
\Mdef{\bO}{\mathbb{O}}
\Mdef{\bP}{\mathbb{P}}
\Mdef{\bQ}{\mathbb{Q}}
\Mdef{\bR}{\mathbb{R}}
\Mdef{\bS}{\mathbb{S}}
\Mdef{\bT}{\mathbb{T}}
\Mdef{\bU}{\mathbb{U}}
\Mdef{\bV}{\mathbb{V}}
\Mdef{\bW}{\mathbb{W}}
\Mdef{\bX}{\mathbb{X}}
\Mdef{\bY}{\mathbb{Y}}
\Mdef{\bZ}{\mathbb{Z}}

\Mdef{\cA}{\mathcal{A}}
\Mdef{\cB}{\mathcal{B}}
\Mdef{\cC}{\mathcal{C}}
\Mdef{\mcD}{\mathcal{D}} % Something funny about \cD.
\Mdef{\cE}{\mathcal{E}}
\Mdef{\cF}{\mathcal{F}}
\Mdef{\cG}{\mathcal{G}}
\Mdef{\mcH}{\mathcal{H}} % There's something funny about \cH: it 
\Mdef{\cI}{\mathcal{I}}
\Mdef{\cJ}{\mathcal{J}}
\Mdef{\cK}{\mathcal{K}}
\Mdef{\mcL}{\mathcal{L}}% There's something funny about \cL: it 
\Mdef{\cM}{\mathcal{M}}
\Mdef{\cN}{\mathcal{N}}
\Mdef{\cO}{\mathcal{O}}
\Mdef{\cP}{\mathcal{P}}
\Mdef{\cQ}{\mathcal{Q}}
\Mdef{\mcR}{\mathcal{R}}% There's something funny about \cR: it 
\Mdef{\cS}{\mathcal{S}}
\Mdef{\cT}{\mathcal{T}}
\Mdef{\cU}{\mathcal{U}}
\Mdef{\cV}{\mathcal{V}}
\Mdef{\cW}{\mathcal{W}}
\Mdef{\cX}{\mathcal{X}}
\Mdef{\cY}{\mathcal{Y}}
\Mdef{\cZ}{\mathcal{Z}}

\Mdef{\ca}{\mathcal{a}}
\Mdef{\ct}{\mathcal{t}}

\Mdef{\At}{\tilde{A}}
\Mdef{\Bt}{\tilde{B}}
\Mdef{\Ct}{\tilde{C}}
\Mdef{\Et}{\tilde{E}}
\Mdef{\Ht}{\tilde{H}}
\Mdef{\Kt}{\tilde{K}}
\Mdef{\Lt}{\tilde{L}}
\Mdef{\Mt}{\tilde{M}}
\Mdef{\Nt}{\tilde{N}}
\Mdef{\Pt}{\tilde{P}}

\Mdef{\tA}{\tilde{A}}
\Mdef{\tB}{\tilde{B}}
\Mdef{\tC}{\tilde{C}}
\Mdef{\tE}{\tilde{E}}
\Mdef{\tH}{\tilde{H}}
\Mdef{\tK}{\tilde{K}}
\Mdef{\tL}{\tilde{L}}
\Mdef{\tM}{\tilde{M}}
\Mdef{\tN}{\tilde{N}}
\Mdef{\tP}{\tilde{P}}

\Mdef{\ft}{\tilde{f}}
\Mdef{\xt}{\tilde{x}}
\Mdef{\yt}{\tilde{y}}

\Mdef{\Ab}{\overline{A}}
\Mdef{\Bb}{\overline{B}}
\Mdef{\Cb}{\overline{C}}
\Mdef{\Db}{\overline{D}}
\Mdef{\Eb}{\overline{E}}
\Mdef{\Fb}{\overline{F}}
\Mdef{\Gb}{\overline{G}}
\Mdef{\Hb}{\overline{H}}
\Mdef{\Ib}{\overline{I}}
\Mdef{\Jb}{\overline{J}}
\Mdef{\Kb}{\overline{K}}
\Mdef{\Lb}{\overline{L}}
\Mdef{\Mb}{\overline{M}}
\Mdef{\Nb}{\overline{N}}
\Mdef{\Ob}{\overline{O}}
\Mdef{\Pb}{\overline{P}}
\Mdef{\Qb}{\overline{Q}}
\Mdef{\Rb}{\overline{R}}
\Mdef{\Sb}{\overline{S}}
\Mdef{\Tb}{\overline{T}}
\Mdef{\Ub}{\overline{U}}
\Mdef{\Vb}{\overline{V}}
\Mdef{\Wb}{\overline{W}}
\Mdef{\Xb}{\overline{X}}
\Mdef{\Yb}{\overline{Y}}
\Mdef{\Zb}{\overline{Z}}

\Mdef{\db}{\overline{d}}
\Mdef{\hb}{\overline{h}}
\Mdef{\qb}{\overline{q}}
\Mdef{\rb}{\overline{r}}
\Mdef{\tb}{\overline{t}}
\Mdef{\ub}{\overline{u}}
\Mdef{\vb}{\overline{v}}

\Mdef{\hc}{\hat{c}}
\Mdef{\he}{\hat{e}}
\Mdef{\hf}{\hat{f}}
\Mdef{\hA}{\hat{A}}
\Mdef{\hH}{\hat{H}}
\Mdef{\hJ}{\hat{J}}
\Mdef{\hM}{\hat{M}}
\Mdef{\hP}{\hat{P}}
\Mdef{\hQ}{\hat{Q}}

\Mdef{\thetab}{\overline{\theta}}
\Mdef{\phib}{\overline{\phi}}

\Mdef{\uA}{\underline{A}}
\Mdef{\uB}{\underline{B}}
\Mdef{\uC}{\underline{C}}
\Mdef{\uD}{\underline{D}}

\Mdef{\bolda}{\mathbf{a}}
\Mdef{\boldb}{\mathbf{b}}
\Mdef{\bfD}{\mathbf{D}}

\Mdef{\fm}{\frak{m}}
\Mdef{\fp}{\frak{p}}
\Mdef{\eps}{\epsilon}

\input{xypic}

\newcommand{\spc}{\mathrm{Spc}}
\newcommand{\fGspecQ}{(\mbox{$G$-spec}/\Q)^c}
\newcommand{\unI}{\mathcal{I}_{un}}
\newcommand{\gI}{\mathcal{I}_{g}}
\newcommand{\Lcl}{\Lambda_{cl}}
\newcommand{\Lqst}{\Lambda_{ct}}
\newcommand{\basic}[1]{G_+\sm_{{#1}}e_{{#1}}S^0}
\newcommand{\elr}[1]{E\langle #1\rangle}
\newcommand{\sub}{\mathrm{Sub}}
\newcommand{\tthick}{\thick_{\tensor}}
\newcommand{\cEi}{\cE^{-1}}
\newcommand{\cOcF}{\cO_{\cF}}
\newcommand{\cOcFK}{\cO_{\cF /K}}
\newcommand{\FC}{FC}
\newcommand{\Nspectra}{\mbox{$N$-{\bf spectra}}}
\newcommand{\Tspectra}{\mbox{$T$-{\bf spectra}}}

\begin{document}
\title{Rational torus-equivariant stable homotopy IV: thick tensor
  ideals and the  Balmer spectrum for finite spectra} 

\author{J.P.C.Greenlees}
\address{School of Mathematics and Statistics, Hicks Building, 
Sheffield S3 7RH. UK.}
\email{j.greenlees@sheffield.ac.uk}
\date{}

\begin{abstract}
We classify thick tensor ideals of finite objects in the category of
rational torus-equivariant spectra, showing that they are completely
determined by geometric isotropy. 

This is essentially equivalent to showing that the Balmer spectrum is
the set of closed subgroups under cotoral inclusion. 
\end{abstract}

\thanks{I am grateful to B.Sanders for conversations about this work.}
\maketitle

\tableofcontents

\section{Introduction}
\subsection{Context}
We are interested in the structural properties of the category of
$G$-equivariant cohomology theories for a compact Lie group $G$. 
To start with, the category is tensor-triangulated, and it is the
homotopy category of the category of $G$-spectra. To see the broad
features of this category we restrict attention to cohomology theories
whose values are rational vector spaces:  a $G$-equivariant cohomology theory with
values in rational vector spaces  is represented by a $G$-spectrum with rational homotopy groups. 

 The category of rational $G$-spectra has a very rich structure. The case when $G$ is a torus
has been studied extensively \cite{tnq1, tnq2, AGs}, and there is a complete algebraic model 
for rational $G$-spectra \cite{tnqcore}.  In this paper we study the
structural features (as a tensor triangulated category) of the 
homotopy category of finite rational $G$-spectra when $G$ is a
torus. From now on the standing assumption is that all spectra are
rational and until Section \ref{sec:beyond} the group $G$ is a rank
$r$ torus. In Sections \ref{sec:beyond} and \ref{sec:toral} we
describe the consequences for more general groups $G$.

\subsection{Overview}
In fact, we give a classification of the finitely generated thick tensor ideals of finite rational
$G$-spectra for a torus $G$. This is in terms of traditional
invariants of transformation groups, namely fixed points and Borel
cohomology.  

Recall that the geometric isotropy of a $G$-spectrum $X$, 
$$\gI (X)=\{ K\st \Phi^KX\not \simeq_1 0\} $$
is the collection of closed subgroups $K$ for which the geometric fixed
points $\Phi^KX$ are non-equvariantly essential.  We say that $L$ is
{\em cotoral} in $K$ if $L$ is a normal subgroup of $K$ and $K/L$ is a
torus. 

\begin{thm} 
(i) If $X$ is a finite rational $G$-spectrum then 
then $\gI (X)$ is closed under passage to cotoral
subgroups and has only finitely many maximal elements. 

(ii) If $X$ and $Y$ are finite rational $G$-spectra with
 $\gI (X)=\gI (Y)$ then $X$ and $Y$ generate the same 
 thick tensor ideal. 
\end{thm}

The closure under passage to cotoral subgroups comes
from the classical Borel-Hsiang-Quillen Localization Theorem. The
rest requires more detailed analysis.

One may also reformulate this in terms of the Balmer spectrum of the
category of rational $G$-spectra. Indeed, the Balmer primes are in 1:1 
correspondence with  the closed
subgroups, containment of primes  corresponds  precisely to cotoral
inclusion and the Balmer support is precisely the geometric isotropy. 

It  is a remarkable vindication of the Balmer spectrum that for the 
groups  considered so far  it  captures the space of subgroups
and  cotoral inclusions,  and even  the f-topology of \cite{ratmack}. This can be put 
down to the fact that both the Balmer spectrum and  the analysis of
rational $G$-spectra are  principally based on the Localization theorem and the
calculation of the rational Burnside ring. For some of the
more delicate arguments it is convenient to use  the very structured packaging of these 
ingredients used in the algebraic model of rational
$G$-spectra \cite{tnq1, tnq2, tnqcore}.

The outline of  the argument is as follows. First, we note that
rationally all {\em natural} cells $G/K_+$ are finitely built from
certain {\em basic}
cells $\sigma_K$ (\cite{tnq1}, see Subsection \ref{subsec:basic}). Any $X$ can be built from basic cells $\sigma_K$ 
for $K$ in the cotoral closure of the geometric isotropy of $X$. This shows that the thick
tensor ideal corresponding to any
admissible collection of subgroups $A$ is generated by the wedge of 
$\sigma_K$ with $K$ maximal in $A$. To complete the proof one must
show that the thick tensor ideal generated by $X$ contains $\sigma_K$ 
for $K$ maximal in the geometric isotropy of $X$; this is the point
where use of the algebraic model plays a central role. 

\subsection{Models}
We are interested in  the tensor-triangulated homotopy category of rational $G$-spectra for a torus $G$, 
and in the tensor-triangulated derived category of an abelian category $\cA (G)$. 

A Quillen equivalence 
$$\mbox{$G$-spectra/$\Q$}\simeq d \cA (G)$$
 is proved in \cite{tnqcore}, where $d\cA (G)$ consists of
 differential graded objects of $\cA (G)$. It follows that there is a triangulated equivalence
$$Ho(\mbox{$G$-spectra/$\Q$})\simeq  D(\cA (G)). $$
For the circle group the Quillen equivalence is shown to be monoidal in \cite{BGKS}, so the triangulated equivalence
preserves the tensor  structure. This is expected to be the case for a general torus $G$, but that is work in 
progress. 

For the present therefore, we need separate treatments for 
 the two tensor-triangulated categories $Ho (\mbox{$G$-spectra/$\Q$})$
 and $D(\cA (G))$. This is irritating, but the arguments we will give
apply equally well to both examples: they  
 only require the existence of a homology functor from the tensor
 triangulated category to $\cA (G)$, and the existence of an Adams
 spectral sequence based on it.

Because of this distinction, we have broken the argument into two parts
for purely expository reasons.  In the first part (Sections \ref{sec:locthm} and \ref{sec:basiccells})
it is easy to give the proof in the conventional language of $G$-spectra, so we give it before 
describing the category $\cA (G)$ in Section \ref{sec:AG}. The
transcription of the first part from the homotopy category of $G$-spectra to  $D(\cA (G))$ is immediate. 
In the second part (Sections \ref{sec:giforfinite} to \ref{sec:exhaust}) it is convenient to use the formal structure of 
$\cA (G)$ so we give the argument in $D(\cA (G))$. Nonetheless, the arguments with the {\em abelian} category 
$\cA (G)$ apply simultaneously to the homology of an object of $d\cA (G)$ and to $\piA_*(X)$ for a $G$-spectrum $X$,
and the triangulated arguments apply to the equivalent categories. The tensor arguments can be conducted 
in either one  of the two structures. 

Readers who feel nervous riding two horses may wish to  fix their attention on $\cA (G)$.
 
\subsection{Contents}

In Section \ref{sec:locthm} we recall some ideas from the classical
theory of transformation groups. In Section \ref{sec:basiccells} we
introduce a  structure of cells and isottropy for studying
rational spectra efficiently.  In Section \ref{sec:AG} we briefly
outline the algebraic model for rational $G$-spectra from \cite{tnq1,
  tnq2, AGs}. Section \ref{sec:giforfinite} is the core of the paper, and
contains the key facts showing that all thick tensor ideals are
determined by the basic cells in them.  In Section \ref{sec:primes} we
recall the definitions of Balmer primes, and show that from this point
of view, the work so far has been a theory of support. In Section
\ref{sec:exhaust} we use the classification of finitely generated
thick tensor ideals to show that all primes are the obvious ones
corresponding to vanishing of geometric fixed points at a closed
subgroup.

We then turn to complementary results. In Section \ref{sec:semifree}
we show that how great an effect the ideal condition has; whereas
the unit $S^0$ gives everything as a thick tensor ideal, we identify 
explicitly the proper thick subcategory it generates in the category of
finite semifree $T$-spectra. Finally, we turn to some
other groups.  In Section \ref{sec:beyond} we deduce from the
case of the circle what happens for $O(2)$ and $SO(3)$, and in Section
\ref{sec:toral} we consider toral spectra for  general groups $G$, and in
particular show that the Balmer spectrum is obtained from that of the
maximal torus by passing to the quotient under the Weyl group action
(the Going Up  and Going Down arguments may be of interest in
themselves).

\subsection{Notation}
Given a partially ordered set ({\em poset}) $X$ and a subset
$A\subseteq X$, we write $\Lambda_{\leq}(A)=\{ b\in X\st b\leq a\in A\}$
for the downward closure of $A$. 

We will consider two partial
orderings on the set of all closed subgroups of a compact Lie group $G$.
We indicate containment of subgroups by $K\subseteq H$, and refer to
this as the {\em classical} ordering. Accordingly $\Lcl (K)$ consists
of all closed subgroups of $K$. The more signifcant ordering in this
paper  is
that of  {\em cotoral inclusion}. We write $K\leq H$
if $K$ is normal in $H$ and $H/K$ is a torus. The importance of this
ordering arises from the Localization Theorem. The set 
$\Lqst (K)$ consists of all closed  subgroups cotoral in $K$.

Given an object $X$ of a triangulated category, we write $\thick (X)$
for the thick subcategory generated by $X$ (i.e., the set of objects
that can be built from $X$ by completing triangles and taking
retracts).  We write $\tthick (X)$
for the thick tensor ideal generated by $X$ (i.e., the set of objects
that can be built from $X$ by completing triangles, tensoring with an
arbitrary object and taking retracts). 

\subsection{Conventions}
 All homology and cohomology will be reduced and have rational coefficients. 

A $G$-equivalence of $G$-spectra will be denoted $X\simeq Y$. A non-equivariant
equivalence of $G$-spectra will be denoted $X\simeq_1 Y$ for emphasis.  

The natural extension to $G$-spectra   of the $H$-fixed point functor on  based  $G$-spaces 
is the geometric $H$-fixed point functor:
$$\Sigma^{\infty}(P^H)\simeq \Phi^H(\Sigma^{\infty}P). $$
Since we routinely omit the suspension spectrum functor, we may write $\Phi^HP$ for the $H$-fixed
point space of a based space $P$ to  avoid ambiguity. 

\section{The Localization Theorem}
\label{sec:locthm}

We revisit some basic facts from transformation groups in our
language. The basic tool is Borel
cohomology, $H^*_G(X):=H^*(EG\times_G X, EG\times_G pt)=H^*(EG_+\sm_GX)$.

The idea is that for finite spectra, geometric isotropy is
determined by Borel cohomology. It then follows from the Localization
Theorem that the geometric isotropy is closed under passage to cotoral
subgroups.

\begin{lemma}
\label{lem:fpBorel}
If $K$ is a torus and $X$ is a  finite $K$-CW-complex, then $X$ is non-equivariantly contractible if and
only if $H^*_K(X)=0$. 
\end{lemma}

\begin{proof}
First, if $X$ is simply connected,  the Hurewicz theorem shows $X\simeq *$ if
and only if $H_*(X)=0$. Since $X$ is finite, that is equivalent to
$H^*(X)=0$. 

Next, we have a fibration $X\lra EK\times_K X \lra BK$ so the Serre spectral
sequence shows that if $H^*(X)=0$ then also $H_K^*(X)=0$. Conversely,
with unreduced cochains the Eilenberg-Moore theorem gives
$$C^*(X)\simeq C^*(EK\times_K X)\tensor_{C^*(BK)} \Q,   $$
(where the tensor product is derived). This shows that $H^*_K(X)=0$ implies $H^*(X)=0$. 
\end{proof}

It follows that $\gI (X)$ can be detected from Borel cohomology of
fixed points. 

\begin{cor}
\label{cor:fpBorel}
If $X$ is finite, $K\in \gI (X)$ if and only if
$H^*_{G/K}(\Phi^KX)\neq 0$.\qqed
\end{cor}

For us the fundamental fact is the following consequence of the Localization Theorem. 

\begin{prop}
\label{prop:gIctclosed}
If $X$ is finite then 
$\gI (X)$ is closed under passage to
cotoral subgroups. 
\end{prop}

\begin{proof}
The Localization Theorem states that if $K$ is a torus and $X$ is a
finite $K$-CW-complex then 
$$H^*_K(X)\lra H^*_K(\Phi^K X )=H^*(BK)\tensor H^*(\Phi^K X)$$
becomes an isomorphism when the multiplicatively closed set $\cE_K =\{
e(W)\st W^K=0\} $ of Euler classes $e(W)\in H^{|W|}(BK)$ is
inverted. The proof uses the fact that the space
$$S^{\infty V(K)}=\bigcup_{W^K=0}S^W, $$
has $H$-fixed points $S^0$ if $H\supseteq K$ and contractible
otherwise. Accordingly, we have a $G$-equivalence
$$X\sm S^{\infty V(K)}\simeq \Phi^KX \sm S^{\infty V(K)}. $$ 

 It follows that if
$H^*(\Phi^KX)\neq 0$ then also $H^*_K(X)\neq 0$. The conclusion
follows from Lemma \ref{cor:fpBorel}.  
\end{proof}

\section{Basic cell complexes}
\label{sec:basiccells}
A distinctive feature of working rationally is that there are many
idempotents in the rational Burnside ring of a finite group. We follow
through the implications of this for cell structures.

Integrally, the relevant cells are homogeneous spaces
$G/K_+$, and the relevant ordering of subgroups is classical containment. 
Rationally, the splitting of Burnside rings means that the
relevant cells are certain basic cells $\sigma_K$ (a retract of
$G/K_+$ introduced below) and the
relevant ordering of subgroups is cotoral inclusion. 

We begin by running through one approach to equivariant cell
complexes, and then introduce basice cells and  follow the same pattern to give the rational
analysis in terms of basic cells.

\subsection{Unstable classical recollections}
Classically, we are used to the idea that based $G$-spaces $P$ are formed from
cells $G/K_+$. The classical unstable  isotropy is defined by 
$$\unI' (P)=\{ K \st P^K \neq
pt\}; $$
 it is not homotopy invariant, but it does have the obvious
property that it is closed under passage to subgroups. 

It is natural to move to a homotopy invariant notion 
$$\unI (P)=\{ K \st  P^K \not \simeq  *\} . $$ 
This notion fits well with the cells we use, since
$$\unI (G/K_+)=\{ L \st L\subseteq K\} =\Lcl(K). $$
Note that we are linking  notions of cell and isotropy with
a partial order on subgroups.

The homotopy invariant version of unstable isotropy may not be closed under passage to subgroups, so that we only
know that $X$ is equivalent to a complex constructed from cells
$G/K_+$ with $K\in \Lcl\unI (T)$, where $\Lcl$
indicates that we take the closure under the classical order (i.e.,
under containment). This can be proved by killing homotopy groups, or
by the method described in the next subsection for the stable
situation. 

\subsection{Stable classical recollections}
Moving to the stable world, for a $G$-spectrum $X$ we have the
stable isotropy
$$\gI (X)=\{ K \st  \Phi^K X \not \simeq_1  0\} , $$
homotopy invariant by definition. Evidently since geometric fixed
points extend ordinary fixed points on spaces,  
$$\unI (P) \supseteq \gI (\Sigma^{\infty}P), $$
and if $P$ can be constructed from cells in $A$ then $\Sigma^{\infty}P
$ can be constructed from stable cells in $A$.

\begin{remark}
In \cite{assiet} and the author's subsequent work this was called {\em
  stable isotropy} to emphasize that stability is the main change from
$\unI$.  The corresponding notion for
categorical fixed points does not seem to be useful, so this caused no
confusion. 

The name of `geometric isotropy' from \cite{HHR} seems to have
acquired currency, and the symmetry between `stable' and `unstable'
does not seem sufficient to overturn this advantage. 
\end{remark}

The attraction of geometric isotropy and geometric fixed points arises from 
the fact that their properties are familiar from the category of based spaces. 
Perhaps the most important instance is that of the Geometric Fixed Point Whitehead 
Theorem, and we state it here because it is fundamental to our approach. The result is
well known to all users of geometric fixed points, and is usually deduced
using isotropy separation to see that that an equivalence in all geometric fixed 
points is an equivalence in all categorical fixed points. 

\begin{lemma} {\em (Geometric Fixed Point Whithead Theorem)}
A map $f: X\lra Y$ of $G$-spectra is an equivalence if $\Phi^Kf: \Phi^KX\lra \Phi^KY$ 
is a non-equivariant equivalence for all closed subgroups $K$.\qqed
\end{lemma}

Next, we note that
$$\gI (G/K_+) =\Lcl(K).  $$
Any $G$-spectrum $X$ can be constructed from stable cells $G/K_+$ with
$K \in \Lcl \gI (X)$. One method is to construct a filtration
analogous to the (thickened) fixed point filtration of a space.
Simplifying this, if $\cF =\Lcl\gI (X)$ then $X\sm
\tilde{E}\cF$ has trivial   geometric fixed points (and is thus contractible by the Geometric Fixed Point
Whitehead Theorem). Now $X \sm E\cF_+$ may be constructed from cells
$G/K_+$ for $K \in \cF$. In effect we use the result for the special
case $E\cF_+$ together with the fact that $G/H_+ \sm G/K_+$ can be
constructed from cells $G/K'_+$ with $K'\subseteq K$.

We now see how we can to take advantage of the additional flexibility of working rationally.

\subsection{Basic cells}
\label{subsec:basic}
The classical cell $G/K_+$  is $S^0$ induced up from $K$, so if the
$K$-equivariant sphere decomposes, so does $G/K_+$. We know $[S^0, S^0]^K=A(K)$ is the
Burnside ring. Rationally this consists of continuous $\Q$-valued
functions on the space $\Phi K$ of conjugacy classes of subgroups of
finite index in its normalizer \cite{tD}. When $K$ has identity component a
torus this means  $A(K)\cong \prod_{(\overline{L})}\Q$, with the product
over conjugacy classes of subgroups $\overline{L}$ of
$\overline{K}=\pi_0(K)$. Accordingly we obtain one primitive
idempotent $e_{\overline{L}}$ for each conjugacy class. 
The building blocks are thus the {\em basic cells} 
$$\sigma_K:=\basic{K}.$$
We  develop cell structures  based on these. The point is that
the geometric isotropy of $\sigma_K$  is smaller than that of $G/K_+$.

\begin{lemma}
\label{lem:gisigmaK} 
The geometric isotropy of $\sigma_K$ consists of all cotoral subgroups
of $K$:
$$\gI (\sigma_K)=\Lqst (K). $$
\end{lemma}

\begin{proof}
The $K$-spectra  $e_LS^0 $ for $L$ a proper subgroup of $K$ can be
constructed with non-fixed $K$-cells, so this is true of the cofibre
of the inclusion $i:e_KS^0\lra S^0$. Inducing up to $G$, the 
cofibre of $i:\sigma_K=\basic{K}\lra G_+\sm_KS^0\simeq G/K_+$ is
constructed from cells $G/L_+$ with $L$ a proper subgroup of $K$. Thus
$\Phi^Ki$ is a non-equivariant equivalence. Since
$\Phi^KG/K_+=G/K_+\not \simeq 0$, $K\in \gI (\sigma_K)$ and hence
by Proposition \ref{prop:gIctclosed}
$\Lqst (K)\subseteq \gI (\sigma_K)$.

Conversely, since $\sigma_K$ is a retract of $G/K_+$, $\gI (\sigma_K)$
consists of subgroups of $K$, and if $L\subset K$ is not cotoral in
$K$ then there is an idempotent $e_L$ orthogonal to $e_K$ with 
$\Phi^L(e_LS^0)=\Phi^L(S^0)$ and hence $\Phi^L(\sigma_K)=
\Phi^L(G_+\sm_Ke_Le_KS^0)\simeq 0$.
\end{proof}

\subsection{Basic detection}
Basic cells play a comparable role to classical cells in that they
generate the category and detect equivalences. The smaller isotropy
means that we can make slightly stronger statements. 

\begin{lemma}
\label{lem:basicbuildsclass}
The category of rational $G$-spectra is generated by the basic cells.
\end{lemma}
\begin{proof}
It suffices to show that the classical cells $G/K_+$ are built from
the basic cells. The proof is by induction on $K$ (i.e., we work with the
poset of all subgroups ordered by inclusion, and note that there are no
infinite decreasing chains). 

Since $G/1_+=\sigma_1$, the induction begins. Now suppose $K$ is
non-trivial  and that 
$G/L_+$ is built from basic cells for proper subgroups $L$ of $K$. Now
$G/K_+$ is a sum of $\sigma_K$ and the spectra $G_+\sm_K e_LS^0$ for
$L\subset K$. The spectrum  $G_+\sm_K e_LS^0$ is built from cells
$G/M_+$ for $M\subseteq K$ as in Lemma \ref{lem:gisigmaK} and hence
from basic cells by induction.
\end{proof}

\begin{remark}
The proof shows that $G/K_+$ is built from basic cells $\sigma_L$ with
$L\subseteq K$.
\end{remark}

There is a useful criterion for vanishing of homotopy in terms of
geometric isotropy. 

\begin{lemma}
\label{lem:disjtsupp}
(i) If $\Lcl (K) \cap \gI (X) =\emptyset$ then $[G/K_+,X]^G_*=0$. 

(ii) If $\Lqst (K)\cap \gI (X)=\emptyset$ then $[\sigma_K,X]^G_*=0$
\end{lemma}

\begin{proof}
The first statement is immediate from the Geometric Fixed Point
Whitehead Theorem, since $\gI (\res^G_KX)=\emptyset$.

For the second, we note 
$$[\sigma_K ,X]^G_*=[e_KS^0, X]^K_*=[e_KS^0, e_KX]^K_*. $$
By hypothesis $\gI (e_KX)=\Lqst (K)\cap \gI (\res^G_K X)=\emptyset$,
so that $e_KX\simeq_K 0$ by the Geometric Fixed Point Whitehead Theorem. 
\end{proof}

A slightly refined version  of  the Whitehead Theorem holds in the
rational context. 

\begin{prop}
\label{prop:basicWHT}
Suppose that $\gI (X), \gI (Y)\subseteq \cK$ and 
$f:X\lra Y$ induces an isomorphism of $[\sigma_K, \cdot ]^G_*$ for all
$K\in \cK$. Then $f$ is an equivalence.  
\end{prop}

\begin{proof}
Taking mapping cones, it suffices to show that if $\gI (Z)\subseteq
\cK$ and  $[\sigma_K ,
Z]^G_*=0$ for all $K\in \cK $ then $Z\simeq 0$.

By Lemma \ref{lem:basicbuildsclass}, it suffices to show $[\sigma_K,
Z]^G_*=0$ for all $K$. There are three cases. If $K\in \gI (Z)$ then
the vanishing follows from the hypothesis on $f$. If $\Lqst (K)\cap \gI
(Z)=\emptyset$, it  follows from Lemma
\ref{lem:disjtsupp}. 

Finally if $K\not \in \gI Z)$ but there is a
subgroup $L \in \Lqst (K) \cap \gI (Z)$, then in view of Proposition  
\ref{prop:gIctclosed} we  may take $L$ to be  finite.  Now, from the hypothesis we have
$$0= [\sigma_L, Z]^G_*=e_L[S^0, Z]^L_*=[S^0, \Phi^LZ]_*.$$
By the Serre spectral sequence we conclude $H^*_{H/L}(\Phi^LZ)=0$ for
every $H$ with $H/L$ a torus. It follows that $\Lqst
(K/L) \cap \gI (\Phi^LZ) =\emptyset$ and $e_{K/L}\Phi^LZ\simeq 0$. Hence
$\Phi^KZ=\Phi^{K/L}\Phi^LZ\simeq 0$. 

Combining the three cases,  $\Lqst K \cap \gI Z=\emptyset$, and therefore
 by Lemma \ref{lem:disjtsupp} $[\sigma_K, Z]^G_*=0$ as required. 
\end{proof}

\subsection{Basic structures}
When we work rationally, classical containment of subgroups is
replaced by cotoral inclusion. Cells $G/K_+$ are no longer
indecomposable, and we have basic cells $\sigma_K$. If $K$ is a torus
then $\sigma_K=G/K_+$, but if $K$ is not a torus then we have a proper inclusion
$$\gI (\sigma_K)=\Lqst (K)\subset \Lcl(K)=\gI (G/K_+). $$

\begin{lemma}
\label{lem:constructfrombasicingi}
Any $G$-spectrum $X$ can be constructed from basic cells $\sigma_K$
with $K$ in $\Lqst (\gI (X))$. 
\end{lemma}

\begin{proof}
Take $\cK =\Lqst \gI (X)$.  We may construct a map $p:P\lra X$ so that $P$ is a wedge of suspensions
of basic cells $\sigma_K$ for $K\in \cK$, and so that $p_*$ is surjective on $[\sigma_K, \cdot
]^G_*$ for all $K \in \cK$. Iterating this, we form a diagram 
$$\diagram
X \ar@{=}[r]&X_0\rto &X_1\rto &X_2\rto &\cdots \\
&P_0\uto &P_1\uto &P_2\uto &
\enddiagram$$

We take $X_{\infty}=\hocolim_s X_s$, and note that 
since $\sigma_K$ is small for each $K$, it follows
 $[\sigma_K, X_{\infty}]^G_*=0$ for $K\in \cK$.
Since  $\gI (X_{\infty})\subseteq \Lqst \gI (X)=\cK$, it follows from
Proposition \ref{prop:basicWHT} that $X_{\infty}\simeq 0$. Arguing with the
 dual tower, we see $X$ can be constructed from cells $\sigma_K$:
 indeed, we define $X^s$ by the cofibre sequence
$$X^s\lra X\lra X_s. $$
By definition  $X^0\simeq *$, and we have
$$\Sigma^{-1}P_s \lra X^s\lra X^{s+1}. $$
Again $X^{\infty}=\colim_s X^s$, and since $X_{\infty}\simeq 0$, we see
$X^{\infty}\simeq X$. 
\end{proof}

\begin{remark}
One can imagine other proofs. One is to construct an analogue of the
map  $E\cF_+ \lra S^0$ for a family $\cF$. This is a spectrum 
$\elr{\cK}$ with geometric isotropy $\cK$ and a map $\elr{\cK}\lra S^0$ which is an equivalence in geometric $K$-fixed points for all
$K\in \cK$ (one construction follows from the results below).  One then mimics the rest of the proof in the classical
case. 

It then follows that  $X\simeq X \sm \elr{\cK}$. 
Now we construct $\elr{\cK}$ out of basic cells $\sigma_K$ with $K
\in \cK$, and claim that $G/H_+\sm \sigma_K$ can be constructed from 
basic cells $\sigma_{K'}$ for $K'$ cotoral in $K$.

It may also be useful to formulate a statement which replaces a
classical cell structure by a basic cell structure with basic cells
$\sigma_K$ for $K$ lying in a cotorally closed collection. 
\end{remark}

The two natural notions of finiteness for rational spectra coincide.

\begin{lemma}
 A rational $G$-spectrum is constructed from finitely many 
basic cells $\sigma_K$ if and only if it is constructed from finitely
many classical cells $G/K_+$.
\end{lemma}

\begin{proof}
Since $\sigma_K$ is a retract of  $G/K_+$, a basic-finite complex is a
finite complex. The standard cell $G/K_+$ is a basic-finite complex
(it is built by basic cells using Lemma \ref{lem:basicbuildsclass},
and then $G/K_+$ is a retract of a finite basic complex using
smallness). Accordingly,  any
classical-finite complex is basic-finite.  
\end{proof}

\section{The abelian model of rational $G$-spectra}
\label{sec:AG}
The main theorem of \cite{tnqcore} states that there is a Quillen equivalence
$$\mbox{$G$-spectra/$\Q$}\simeq d\cA (G)$$
between rational $G$-spectra and differential objects of $\cA (G)$. 
So far we have worked directly in the category of rational $G$-spectra, but for 
the most delicate part of the proof it is convenient to use the algebraic model. 

Accordingly we summarize the definition and properties of the abelian category
$\cA (G)$ that we need from  \cite{tnq1} including an Adams spectral sequence based on it. 
 The present
account is very brief and readers may need to refer to \cite{tnq1} for
details. The structures from that analysis will be relevant to much of what
we do here.

\subsection{Definition of the category}
First we must construct the category $\cA (G)$, which is a  category of modules
over a diagram of rings.  For a category $\bfD$ and a diagram of $R
:\bfD \lra \rings$ of rings,   an $R$-module is given by a $\bfD$-diagram $M$ such that $M(x)$ is an
$R(x)$-module for each object $x$ in $\bfD$, 
and for every morphism $a: x\lra y$ in $\bfD$, the map $M(a): M(x) \lra M(y)$ 
is a module map over the ring map $R(a): R(x) \lra R(y)$.

The shape of the diagram for $\cA (G)$ is given by the partially ordered set $\sub_c(G)$ 
of connected subgroups of $G$.  To start with we consider the single
graded ring
$$\cOcF =\prod_{F\in \cF }H^*(BG/F), $$
where the product is over the family $\cF$ of finite subgroups of $G$. 
To specify the value of the ring at a connected subgroup $K$,  
we use Euler classes: indeed if $V$ is a representation of $G$
we may define $c(V) \in \cO_{\cF}$ by specifying its components. 
In the factor corresponding to the finite subgroup $F$ we take
$c(V)(F):=c_{|V^F|}(V^F) \in H^{|V^F|}(BG/F)$ where $c_{|V^F|}(V^F)$ is the  classical Euler class
of $V^H$ in ordinary rational cohomology.

\newcommand{\cOtcF}{\widetilde{\cO_{\cF}}}
The diagram of rings $\cOtcF$ is defined by the following functor on $\sub_c(G)$ 
$$\cOtcF (K)=\cEi_K \cOcF$$
where $\cE_K =\{ c(V) \st V^K=0\} \subseteq \cOcF$ is the multiplicative
set of Euler classes of $K$-essential representations. Each of the Euler
classes is a finite sum of mutually orthogonal homogeneous terms, and so this localization is again a graded ring. 
Next we consider the category of modules $M$ over the diagram $\cOtcF$. 
Thus the value $M(K)$ is a module over $\cEi_K\cOcF$, and if
$L\subseteq K$, the structure map 
$$M(L)\lra M(K)$$
is a map of modules over the map 
$$\cEi_L \cOcF \lra \cEi_K \cOcF$$
of rings.   Note this map of rings is a localization since
$V^L=0$ implies $V^K=0$ 
so that $\cE_{L}\subseteq \cE_{K}$. 
The category $\cA (G)$ is formed from a subcategory of the
category of $\cOtcF$-modules by adding structure. There are two requirements which
we briefly indicate here.  Firstly they must 
be {\em quasi-coherent}, in that they are determined by their 
value at the trivial subgroup $1$ by the formula 
$$M(K):=\cEi_K M(1). $$

The second condition involves the relation between $G$ and its quotients. 
Choosing a particular connected subgroup $K$, we consider the
relationship between the group
$G$ with the collection $\cF$ of its finite subgroups  
and the quotient group $G/K$ 
with the collection $\cF /K$ of its finite subgroups.  
For $G$ we have the ring $\cOcF$ and for $G/K$ we have 
the ring
$$\cOcFK =\prod_{\tK \in \cF /K}H^*(BG/\tK)$$
where we have identified finite subgroups of $G/K$ with 
their inverse images in $G$, i.e., with subgroups $\tK$ of $G$
having identity component $K$. Combining the inflation maps associated
to passing to quotients by $K$ for individual groups, there is an inflation map 
$$\cOcFK \lra \cOcF. $$
%that we consider in more detail below. %in Section \ref{sec:PKDEFpvsDEFKp} below.
The second condition is  that the object should be {\em extended}, in 
the sense that for each connected subgroup $K$ there is a specified isomorphism 
$$M(K) \iso \cEi_K \cOcF \otimes_{\cOcFK} \phi^K M$$
for some $\cOcFK$-module $\phi^KM$, which is  a given part of the structure. 
These identifications should be compatible  when we have inclusions of connected
subgroups.  If we choose a subgroup $L$ and then the modules $\phi^KM$ for
$K\supseteq L$ fit together to make an object of $\cA (G/L)$.

\subsection{Connection with topology}
The connection between $G$-spectra and $\cA (G)$
is given by  a homotopy functor 
$$\piA_*: \Gspectra \lra \cA (G)$$
with the exactness properties of a homology theory. 
It is rather easy to write down the value of the functor as
a diagram of abelian groups.
 
\newcommand{\efp}{E\cF_+}
\newcommand{\siftyV}[1]{S^{\infty W(#1)}}
\begin{defn}
\label{defn:piA}
For a $G$-spectrum $X$ we define $\piA_*(X)$ on $K$  by 
$$\piA_*(X)(K)=\pi^G_*(D \efp \sm \siftyV{K} \sm X).$$
Here $\efp$ is the universal space for the family $\cF$ of finite 
subgroups with a disjoint basepoint added and $\DH \efp =
F(\efp , S^0)$ is its functional dual
(the function $G$-spectrum of maps from $\efp$ to $S^0$).
The $G$-space $\siftyV{K}$ is defined by  
$$\siftyV{K} =\colim_{V^K=0} S^V, $$
when $K \subseteq H$, so there is a map 
$\siftyV{K} \lra \siftyV{H}$ inducing the map 
$\piA_*(X)(K)\lra \piA_*(X)(H)$.\qqed
\end{defn}

The  definition of $\piA_*(X)$ shows that quasi-coherence for 
$\piA_*(X)$ is just a matter of understanding Euler classes. 
The extendedness of  $\piA_*(X)$ is a little more subtle, and will 
play a significant role later.  We take 
$$\phi^K \piA_*(X)= \pi^{G/K}_*(D E\cF/K_+ \sm \Phi^K(X)),$$ 
where $\PK$ is the geometric fixed point functor, and the extendedness
 follows from properties of the geometric
fixed point functor.

To see that $\piA_*(X)$ is a module over $\cO$, the key is to understand $S^0$.
 
\begin{thm}
\cite[1.5]{tnq1}
The image of $S^0$ in $\cA (G)$ is the structure functor:
$$\widetilde{\cO}_{\cF} =\piA_*(S^0), $$
with the canonical structure as an extended module. 
\end{thm}

Some additional work confirms that $\piA_*$ has the appropriate behaviour.  
\begin{cor}
\cite[1.6]{tnq1}
The functor $\piA_*$ takes values in the abelian category $\cA (G)$.
\end{cor}

\subsection{Geometric fixed points, geometric isotropy and basic cells}

Since the model was built on the idea that geometric fixed points are fundamental it is easy 
to read off the algebraic counterparts. 

\begin{remark}
The counterpart of the geometric $K$-fixed points in $\cA (G)$ is obtained by restricting attention to 
subgroups over $K$. More precisely, if $H\supseteq K$ then 
$$\phi^{H/K}(\Phi^KX) =\phi^H X, $$
or equivalently
$$X(H)=\cEi_H\cOcF \tensor_{\cO_{\cF /K}} \cEi_{H/K}\cO_{\cF/K} \tensor_{\cO_{\cF/H}}\phi^HX=\cEi_H\cOcF \tensor_{\cO_{\cF /K}} (\Phi^KX)(H/K). $$
\end{remark}

\begin{remark}
The geometric fixed points of an object $X$ of $d\cA(G)$ are determined in terms of the algebraic model by 
$$\gI (X)=\{ K\st H_*X(K)\neq 0\}. $$
\end{remark}

\begin{remark}
If $\tilde{K}$ is a subgroup with identity component $K$ then 
the algebraic model of the basic cell $\sigma_{\tilde{K}}$ is $f_{K}(\Q_{\tilde{K}})$. In other words, it is determined by its value at the identity component of $K$, 
and at that point it is extended from the $\cO_{\cF/K}$-module $\Q$ pulled back from $H^*(BG/\tilde{K})$. In fact, the value is $\cEi_K\cOcF \tensor_{H^*(BG/\tilde{K})}\Q$. 

Its value at $L$ is zero unless $L\subseteq K$, where  the value is $\cEi_L\cOcF\tensor_{H^*(BG/\tilde{K})}\Q$. 
\end{remark}

\subsection{The Adams spectral sequence}
The homology theory $\piA_*$ may be used as the basis of an 
Adams spectral sequence for calculating maps between rational
$G$-spectra. The main theorem of \cite{tnq1} is as follows.

\begin{thm} (\cite[9.1]{tnq1})
For any rational $G$-spectra $X$ and $Y$ there is a natural
Adams spectral sequence
$$\Ext_{\cA (G)}^{*,*}(\piA_*(X) , \piA_*(Y))\Rightarrow [X,Y]^{G}_*.$$
It is a finite spectral sequence concentrated in rows $0$ to $r$ (the rank of $G$)
and strongly convergent for all $X$ and $Y$. \qqed
\end{thm}

\section{Geometric isotropy of finite spectra}
\label{sec:giforfinite}
This section is the core of the entire paper. We show that thick
tensor ideals are completely determined by the basic cells they
contain. In one direction, we show that if $K$ is maximal in the
geometric isotropy of $X$, then $\sigma_K$ is in $\tthick (X)$, and in
the other that basic cells construct everything that their geometric
isotropy suggests. This comes down to a detailed calculational
understanding of maps out of basic cells.

The following finiteness theorem is fundamental. 

\begin{lemma}
\label{lem:finmax}
If $X$ is a finite spectrum then $\gI (X)$ has finitely many maximal
elements. 
\end{lemma}

\begin{proof}
By Lemma \ref{lem:constructfrombasicingi}, we may suppose that $X\simeq X' $ where $X'$ is constructed from
basic cells $\sigma_K$ with $K$  from $\Lqst \gI (X)$. Since $X\simeq
X'$, the identity factors as $X\lra X'\lra X$, and since $X$ is small,
we see $X$ is a retract of a finite subcomplex of $X'$. In other
words, $X$ is a retract of  a finite complex using cells
$\sigma_{K_1}, \ldots , \sigma_{K_N}$ with $K_i\in \Lqst \gI (X)$. 
and hence
$$\supp (X)\subseteq \bigcup_i \Lqst (K_i). $$

Forming the dual tower, we see that $X$ may be converted to a point
using the basic cells $\sigma_{K_1}, \ldots , \sigma_{K_N}$. Elements
$K$ of $\gI (X)$  can only be removed if one of the cells is
$\sigma_H$ with $K\in \Lqst H$, and hence the maximal elements
of $\gI (X)$ must be amongst the $K_i$.
\end{proof}

\begin{lemma}
If $A$ is a set of primes with finitely many maximal elements
there is a finite spectrum with $\gI (X)=A$.
\end{lemma}

\begin{proof}
We simply enumerate the maximal elements $K_1, \ldots , K_N$ and take
$X=\sigma_{K_1}\vee\cdots \vee \sigma_{K_N}$. The result follows from
Lemma \ref{lem:gisigmaK}. 
\end{proof}

\begin{lemma}
\label{lem:tthicksigmaK}
 If $\gI (X)\subseteq \Lqst (K_1) \cup \cdots \cup \Lqst (K_N)$
then 
$X\in \tthick(\sigma_{K_1}, \ldots , \sigma_{K_N})$. 
\end{lemma}

\begin{proof}
It suffices to show we can kill top isotropy. Indeed, if $K$ is
maximal in $\gI (X)$ we must show that the isotropy can be killed by
attaching finitely many cells $\sigma_K$. In fact, since
$K$ is maximal,  $T=H^*_{G/K}(\Phi^KDX)$  is torsion, and since it is
also finitely generated, this is a finite
dimensional vector space. We show that this dimension can be reduced
by adding a basic cell $\sigma_K$. 

Indeed, we may define $X'$ by the cofibre sequence
$$X\stackrel{\alpha}\lra f_K(T)\lra X', $$
where $\alpha$ is the identity at $K$. Accordingly, 
 $K$ is not in $\gI (X')$. Picking $t\in T$ of lowest degree $n$
we see there is a map $\alpha: \sigma_K^n \lra f_K(T)$ hitting
$t$. Indeed, we consider the Adams spectral sequence for calculating
$[\sigma_K, f_K(T)]^G_*$. Its $E_2$-term is easy to understand, since
$$\Ext_{\cA (G)}^{*,*}(\piA_*(\sigma_K), \piA_*(f_K(T))
=\Ext_{H^*(BG/K)}^{*,*}(\piA_*(\sigma_K)(K), T)=
\Ext_{H^*(BG/K)}^{*,*}(\Q, T),  $$
and we see by degree that $t\in \Hom_{H^*(BG/K)}^{*}(\Q, T)$ is an infinite cycle. 

\begin{lemma}
  Since $K\not \in \gI (X')$ the module 
$[\sigma_K, X']$ is annihilated by $\cE_K$
\end{lemma}

\begin{proof}
First note that by the nature of injectives, the property that $X'(K)=0$ is inherited by its injective resolution. 
It therefore suffices to show $\Hom (\sigma_K, Y')$ is torsion when $Y'(K)=0$. Next, note that a map from $\sigma_K$ 
is determined by its value at 1. 
 Now consider the square
$$\begin{array}{ccc}
\sigma_K(K)&\lra & Y'(K)=0\\
\uparrow&&\uparrow\\
\sigma_K(1)&\lra &Y'(1)
\end{array}$$
It follows that the image of each element of $x\in \sigma_K(1)$ is annihilated by an Euler class $e(W_x)$ with $W_x^K=0$. 
Since $\sigma_K(L)$ is finitely generated, there is a single representation $W$  independent of $x$ with $W^K=0$ so that $e(W)$
annihilates the whole image of $\sigma_K(1)$
\end{proof}

By the lemma,  we may find a
$K$-essential representation $W$ so that the right hand diagonal composite 
$$\diagram 
&S^{-W}\sm \sigma_K^n\dto^{e(W)} \ddrto^0 \ar@{.>}[ddl]_{\widetilde{e(W)\alpha}}
&\\
&\sigma^n_K\dto^{\alpha}&\\
X\rto &f_K(T)\rto &X'
\enddiagram$$
is zero, and so $e(W)\alpha$ lifts to a map  $\widetilde{e(W)\alpha}$ to $X$. 

We note that $\alpha$ is nonzero as a map  $\Q \lra T$ (i.e., more precisely, this is the map $\phi^K\alpha$), 
and hence the same is true of $\widetilde{e(W)\alpha}$. Accordingly, if we take $X_1=\cofibre
(\widetilde{e(W)\alpha})$ the vector space $H^*_{G/K}(\Phi^KDX_1)$ is of smaller dimension. After adding 
finitely many cells we reach $X_N$ with
$H^*_{G/K}(\Phi^KDX_1)=0$, so that $K\not \in \gI (X_N)$, and we
have a cofibre sequence
$$X\lra X_N \lra C$$ 
where $C$ is built from suspensions of copies of $\sigma_K$.

By  the triangle property of geometric isotropy, 
$$\gI (X_N)\subseteq \gI (X)\cup \Lqst (K)=\gI (X). $$
In finitely many steps we may use the above to
 remove all top dimensional geometric isotropy. Repeating this
 finitely many times we are left with empty geometric isotropy and
 hence a contractible object by the Geometric Fixed Point Whitehead Theorem. 
\end{proof}

\begin{thm}
\label{thm:sigmaKintthick}
If $K$ is maximal in $\gI (X)$ then $\sigma_K\in\tthick (X)$
\end{thm}

\begin{proof}
For the particular group $G$ we will argue by induction on the
dimension of the subgroup $K$ maximal in $\gI (X)$. If $K$ is 0
dimensional then $f_1(X(K))$ is a retract of $X$, and $X(K)$ is a
torsion  module over $H^*(BG/K)$. It is well known that any two
finitely generated torsion modules over $H^*(BG/K)$ generate the same thick subcategory,
so $\Q$ is finitely built by $X(K)$ and $\sigma_K$ is finitely built
by $f_1(X(K))$. 

Now suppose $K$ is of dimension $d$   and that the result is
true for those of smaller dimension.  
If $L$ is a proper cotoral subgroup of $K$ then $L$ is a maximal
subgroup of $\gI (X\sm \sigma_L)$ of smaller dimension, and hence  by induction 
$\sigma_L\in \tthick (X\sm \sigma_L)\subseteq \tthick (X)$. We will
use this Inductive Observation several times below.  It remains
to show that $\sigma_K$ is in $\tthick (X)$.

As a warm-up, we will first consider the special case $K=G$ (those
fluent with $\cA (G)$ can skip straight to the general case). 
Now $X(G)=\cEi_G\cOcF\tensor V$ for a graded vector space $V$, and
there is a map $i: X\lra f_G(V)$ which is the identity at $G$. However
$f_G(V)$ is not a finite spectrum. We
claim that the image of $i$ lies inside a wedge of  spheres $WS$.

Note that since $X$ is finite, $V$ is finite dimensional and we may
write $V=\bigoplus_i \Sigma^{n_i}\Q$.  The most obvious finite spectrum $Y$ with 
$\Phi^GY=V$ is $\bigvee_i S^{n_i}$, but typically the image of $X$
will not be a subspace of this.  However, if we suspend this wedge of
spheres by any representation $W$ with $W^G=0$ we obtain another candidate. 
On the other hand, 
the quotient of $f_G(V)(1)=\cEi_G\cOcF \tensor V$ by $iX(1)$ 
is $\cE_G$-torsion, and by finiteness of $X$ there is a single complex
representation $W$ so that the Euler class $c(W)$ annihilates it. 
and the sum is finite since $X$ (and hence $\Phi^GX$) is finite. 

Hence we find
$$i(X)\subseteq S^{-W} \sm \bigvee_i S^{n_i} =:WS. $$
We may now consider the cofibre sequence
$$X\lra WS\lra Y. $$
Since $X$ and $WS$ are small so is $Y$.  Since $X \lra WS$ is an
isomorphism  at $G$, the group $G$ is not in the geometric isotropy of
$Y$, so $Y$ is constructed from  $\sigma_L$ with $L$ proper. By the
Induction Observation these basic cells can be built from $X$, so  $WS$ is in $\tthick (X)$. 
Any one of the spheres $S^{n_i-W}$ is a retract of $WS$ and a
suspension of it is in the thick tensor ideal, so that 
$$S^0 \in \tthick (WS) \subseteq \tthick(X). $$

Now we turn to general case.  First, $X(K)=\cEi_K\cOcF\tensor_{\cOcFK} T$ for a torsion
$\cOcFK$-module $T$. By the maximality of $K$ and the adjointness
property of $f_K$, there is a map $i: X\lra f_K(T)$ which is the
identity at $K$. However $f_K(T)$ will typically not be a finite
spectrum.

We construct a finite complex $\FC'$ from suspensions of copies of
$\sigma_K$ to play the role of $f_K(T)$. Indeed, we use the  the procedure of
Lemma \ref{lem:tthicksigmaK} to construct  a map $f': \FC' \lra f_K(T)$ which is an isomorphism 
on $\phi^K$. The quotient of $f_K(T)(1)=\cEi_K\cOcF\tensor_{\cOcFK}T$ by $iX(K)$ is $\cE_K$-torsion, 
so there is a representation $W$ with $W^K=0$ annihilating it. Accordingly we may take $\FC=\Sigma^W\FC'$
and obtain a map  $f: \FC=\Sigma^W\FC' \lra f_K(T)$ which is the identity at $K$ and with 
$\im (i) \subseteq \im (f)$. 

We may now consider the cofibre sequence
$$X\lra \FC \lra Y. $$
Since $X$ and $\FC $ are small so is $Y$.  Since $X \lra \FC$ is an
isomorphism  at $K$, the group $K$ is not in the geometric isotropy of
$Y$, so $Y$ is constructed from  $\sigma_L$ with $L$ a proper cotoral
subgroup of $K$. By the Inductive Observation, these basic cells can be built from $X$, so  $\FC$ is in
$\tthick (X)$. 

Finally, $\sigma_K \in \tthick (\FC)$. This can be proved in various ways, and we will use Hopkins's method
 \cite{Hopkins}. Suppose the top dimensional
basic  cells of $\FC$ are of dimension $n$. They are all attached to 
the basic $(n-1)$-skeleton, so we can add them in any order. In
particular, there is a map $\FC \lra \sigma^n_K$ which is surjective
in $H_*$. Accordingly the mapping cone $f:\sigma^n_K \lra R$ is zero in homology and hence
tensor nilpotent. Now $\Sigma \FC =C(f)$ so $C(f)\tensor \sigma_K \in \tthick(\FC)$, and hence
$C(f^n)\tensor k \in \tthick (\FC)$. If $f^n\simeq 0$ then  $\sigma_K$ is a retract of $C(f^n)$. 
\end{proof}

\begin{cor}
\label{cor:tthick}
 If $X$ and $Y$ are finite spectra with $\gI (X)=\gI (Y)$ then 
$\tthick (X)=\tthick(Y)$. \qqed
\end{cor}

\begin{proof}
 If $K_1, \ldots , K_N$ are the maximal elements of $\gI (X)$ then
$\sigma_{K_i} \in \tthick (X)$ by Theorem \ref{thm:sigmaKintthick}. 
Hence $\tthick (X) \supseteq \bigcup_i \Lqst (K_i)$ by Lemma
\ref{lem:tthicksigmaK}. The reverse inclusion is obvious. 
\end{proof}

Some may prefer the following reformulation. A thick tensor ideal is
called {\em finitely generated} if it is generated by a finite number
of  small spectra (or equivalently, by just one). 

\begin{cor}
\label{cor:geomiso}
The finitely generated thick tensor ideals of finite rational $G$-spectra are precisely
the fibres of geometric isotropy
$$\gI : \mbox{finite-rational-$G$-spectra}\lra \mathcal{P}(\sub
(G)); $$
the image consists of collections of subgroups with finitely many
cotorally maximal elements that are
closed under cotoral specialization. 
\end{cor}

\begin{remark}
The finite generation hypothesis is necessary, since for example the
collection of finite $\cF$-spectra is a thick tensor ideal but its
geometric isotropy has infinitely many maximal elements. 
\end{remark}

%\begin{lemma}
%\label{lem:finchains}
%The length of any cotoral chain in $G$ is bounded  by the rank. 
%\end{lemma}

\section{Primes}
\label{sec:primes}

It is interesting to reformulate our classification in terms of Balmer
prime spectra. In fact, the calculation of Balmer spectra was the
initial objective of this project, but it only made progress when
formulated in terms of geometric isotropy and thick tensor ideals. In
effect this shows the power of considering supports. 

\subsection{Background}
We recall the basic definitions from \cite{Balmer1}. 

\begin{defn}
A {\em prime ideal} in a tensor triangulated category is a thick
proper tensor ideal $\wp$ with the
property that $ab\in \wp$ implies that $a$ or $b$ is in $\wp$. 

The {\em Balmer spectrum} of a tensor-triangulated category $\C $ is $\spc (\C)=\{ \wp \st \wp \mbox{
  is prime } \}.$ The Zariski topology on $\spc (\C)$ is the one
 generated by the closed sets $\supp (X)=\{ \wp \st X\not \in \wp\}$ as $X$ runs through objects of $\C$. 
\end{defn}

\begin{thm}
If $G$ is a torus then 
$\spc (\fGspecQ)$ is the partially ordered set
$\sub_a(G)$ whose elements are closed subgroups with $L\leq K$ if $L\subseteq K$
with $K/L$ a torus. 
\end{thm}

It is rather easy to find $\sub_a(G)$ inside $\spc (\fGspecQ)$. The
main work will be to show that we have found all the primes. 

To start with we consider 
$$\wp_K=\{ X\st \Phi^KX\simeq_1 0 \}. $$
To see this is prime we note that $0$ is a prime in rational spectra
(since it is equivalent to chain complexes of $\Q$-modules) and 
$$\wp_K=(\Phi^K)^*((0)) \mbox{ where } \Phi^K: \mbox{$G$-spectra}\lra
\mbox{spectra}. $$

\begin{lemma}
For any finite spectrum $X$, the support in the sense of Balmer for
this set of primes coincides with the geometric isotropy:
$$\supp (X)=\{ H \st X\not \in \wp_H \}=\{ H \st \Phi^HX \not  \simeq
0\}=\gI (X). \qqed $$
\end{lemma}

We will in due course show that the $\wp_K$ give all primes, so that
this will be the full Balmer support. 

\begin{lemma}
The intersection of all these primes is trivial 
$$\bigcap_{K}\wp_K=0.$$
\end{lemma}

\begin{proof}
This says that any spectrum with trivial geometric
fixed points for all subgroups is contractible, which is the Geometric
Fixed Point Whitehead Theorem. 
\end{proof}

We should note that these tensor ideals are not finitely generated. 
\begin{lemma}
The geometric isotropy of $\wp_K$ is 
$$\gI (\wp_K)=\{ H \st K\not \leq H\}=\bigcup_{K\not \leq
  H}\Lambda_{ct}(H). $$
Thus $\sigma_H\in \wp_K$ whenever $K\not \leq H$.\qqed
 \end{lemma}

\subsection{The lattice of subgroup primes}
The Localization Theorem shows that containment of subgroup primes
precisely corresponds to cotoral inclusion.

\begin{lemma}
 $\wp_L \subseteq \wp_K$ if and only if $L\subseteq K$ with $K/L$
a torus. 
\end{lemma}

\begin{proof}
First suppose $L$ is cotoral in $K$. We need to show that
$\Phi^LX\simeq 0$ (non-equivariantly)  implies $\Phi^KX\simeq 0$
(non-equivariantly). This  follows from the Localization
Theorem. In more detail, we may as well consider the $K/L$-spectrum
$Y=\Phi^LX$.  The hypothesis is that $Y$ is nonequivariantly
contractible (ie that $EK/L_+\sm Y\simeq 0$) and the conclusion is
that $\Phi^{K/L}Y$ is nonequivariantly contractible. However the
localization theorem states that for finite $K/L$-complexes $Z$ the map
$$H_{K/L}^*(Z)\lra  H_{K/L}^*(\Phi^{K/L}Z)=H^*(BK/L)\tensor H^*(\Phi^{K/L}Z)$$
becomes an isomorphism when $\cE_{K/L}$ is inverted. Taking $Z=DY$
gives the required result since $\cEi_{K/L}H^*(BK/L)$ is nonzero. 

If $L$ is not a subgroup of $K$ then $G/K_+$ has trivial $L$-fixed
points and non-trivial $K$-fixed points so $\wp_L\not \subseteq
\wp_K$. Finally,  suppose $L$ is a subgroup of $K$, but not cotoral. 
We may consider the $K$-spectrum $\sigma_K$; by Lemma \ref{lem:gisigmaK} 
$\Phi^K\sigma_K\not\simeq 0$ but $\Phi^L\sigma_K\simeq 0$ so 
$\wp_L\not \subseteq \wp_K$.
\end{proof}

\section{The subgroup primes exhaust the primes}
\label{sec:exhaust}
Finally, in this section we deduce the structure of the Balmer
spectrum from our classification of thick tensor ideals. The main
point is to show that all primes are of the form $\wp_H$ for some
closed subgroup $H$.

\subsection{Thick tensor ideals and primes}
The relationship between thick tensor ideals and primes is easily
deduced from the classification of thick tensor ideals Corollary \ref{cor:tthick}.  

\begin{cor}
The thick subcategory generated by a finite rational spectrum $X$ is
an intersection of primes: 
$$\tthick(X)=\bigcap_{K\not \in \gI (X)}\wp_K.\qqed $$
\end{cor}

\subsection{Exhaustion}
We finally want to prove that we have found all the primes. 

\begin{thm}
\label{thm:allprimes}
Every prime $\wp$ of $\fGspecQ$ is $\wp_K$ for some closed subgroup $K$.
\end{thm}

This essentially follows from Corollary \ref{cor:geomiso}. We first
observe that it suffices to show that an arbitrary prime is an
intersection of those of the form $\wp_L$.

\begin{lemma}
\label{lem:primenotint}
If $\wp $ is a prime and $\wp =\bigcap_{L\in A}\wp_L$ then $A$ has a
unique minimal element $L$ and $\wp=\wp_L$.
\end{lemma}

\begin{proof}
If not we can choose a subgroup $L$ in $A$ which is not
redundant. Thus 
$$\wp =\wp_L \cap \bigcap_{K\in A\setminus \{ L\} } \wp_K, $$
and since $\wp_L$ is not redundant, we may choose $X_L\in \wp_L\setminus \wp$ and 
$Y_L\in \bigcap_{K\in A\setminus \{ L\} } \wp_K\setminus \wp$. 
%These exist by Theorem \ref{thm:geomiso} since we can construct finite complexes with any geometric
%isotropy closed under passage to passage to subtoral subgroups, 
This contradicts the fact that $\wp$ is prime since
$X_L\sm Y_L \in \wp$ but $X_L\not \in \wp$ and $Y_L\not \in \wp$.
\end{proof}

\begin{proof}[of \ref{thm:allprimes}]
Write $\gI (\wp)=\bigcup_{X\in \wp}\gI (X)$. This is a countable
set, so we can list its elements $K_1, K_2, \ldots$. If $K\in \gI
(\wp)$ then $\Lqst (K)=\gI (\sigma_K) \subseteq \gI (\wp)$
and hence $\sigma_K\in \wp$.

Hence 
$$\wp =\bigcup_i \thick (\sigma_{K_i})$$
Since
$$\thick(\sigma_K)=\bigcap_{L\not \leq K} \wp_L$$
$$\wp =\bigcap_{L\not \in \bigcup_i \Lqst (K_i)}\wp_L$$
By Lemma \ref{lem:primenotint}  if there is more than one minimal prime in this list $\wp $ is not prime. 
\end{proof}

\subsection{The Zariski topology}
We observe that the topology is entirely generated by
containments. Noting that Balmer primes reverse the inclusions of
commutative algebra  primes, closed sets are closed under passage to smaller primes. 
 
\begin{lemma}
The Zariski topology on $\spc (\fGspecQ)$ is generated by the closed sets $\Lambda (\wp_H)$. 
\end{lemma}

\begin{proof}
Since $\gI (\sigma_K)= \Lqst (K)$, we see 
$\Lambda (\wp_H)=\{ \wp_K \st K \leq H\} =\supp (\sigma_K)$, and the specified subsets are closed. 

Again, for a finite spectrum $X$, Lemma \ref{lem:finmax}  shows $\supp (X)$ has finitely many maximal elements
$K_1, \ldots , K_N$, so that $\supp (X)=\supp (\sigma_{K_1})\cup \cdots \cup \supp (\sigma_{K_N})$. 
\end{proof}

\section{Semifree $\protect \T$-spectra}
\label{sec:semifree}
The point of this section is to show that it is much harder to
classify thick subcategories than thick tensor-ideals. It will suffice
to look at semifree $\T$-spectra, i.e., those $\T$-spectra with $\gI
(X)\subseteq \{ 1, \T\}$. The model for these \cite{s1q}  is sufficiently simple
that we may be explicit. 

\subsection{The model of semifree $\protect\T$-spectra}
The  model $\cA_{sf}(T)$ of semifree $T$-spectra can be obtained from the model $\cA (T)$ of all $T$-spectra
by restriction, but it is easier to repeat the construction from scratch.  
In fact $\cA_{sf}(T)$ is the abelian category of objects $\beta : N \lra \Q[c,c^{-1}]\tensor V$
where $N$ is a $\Q[c]$-module, $V$ is a graded $\Q$-vector space and $\beta$ is the $\Q[c]$
map inverting $c$. In effect, we have  the $\Q [c]$-module $N$, together with a chosen `basis'
$V$. Morphisms are commutative squares
$$\diagram 
M\rto^{\theta} \dto&N\dto\\
\Q[c,c^{-1}]\tensor U \rto^{1\tensor \phi}&\Q[c,c^{-1}]\tensor V
\enddiagram$$
The category $\cA_{sf}(T)$ is of injective dimension 1, and the ring $\Q[c]$ is evenly graded, so every object of 
$d\cA_{sf}(T)$ is formal, and we will identify semifree $T$-spectra  $X$ (or objects of $d\cA_{sf}(T)$) with their
image $\piA_*(X)$ in the abelian category $ \cA_{sf}(T)$.

The fact we are talking about {\em ideals} is essential for Corollary \ref{cor:geomiso}. If we consider
semifree $G$-spectra when $G$ is the circle then there are just two thick tensor
ideals of finite spectra
\begin{itemize}
\item  free specta (with geometric isotropy $1$, generated by $G_+$) 
\item  all spectra (with geometric isotropy $\{ 1, G\}$, generated by $S^0$). 
\end{itemize}
On the other hand, the thick subcategory generated by $S^0$ (without the ideal
property) does not contain $G_+$, and we will make it explicit. The classification of 
thick subcategories in general seems complicated, and we do not give a complete answer. 

\subsection{Wide spheres}
The small objects with $\beta$ injective are the objects $X=(\beta:
N\lra \Q[c,c^{-1}]\tensor V )$ with $\beta$ injective, $V$ finite
dimensional and $N$ bounded above; these objects are called {\em wide
  spheres} \cite{s1q}. 

We note that $\Q [c,c^{-1}]\tensor V$ is the same in each even degree
and the same in each odd degree. We therefore let
$$|V|_0=\bigoplus_{k}V_{2k} \mbox{ and }
|V|_1=\bigoplus_{k}V_{2k+1}. $$
We will fix  isomorphisms 
$$|V|_0\cong (\Q [c, c^{-1}]\tensor V)_0\mbox{ and }|V|_1\cong (\Q [c,
c^{-1}]\tensor V)_1, $$
and then use multiplication by powers of $c$ to identify other graded
pieces of $\Q [c,c^{-1}] \tensor V$ with the appropriate one. 

We will want to think of stepping down the degrees in steps of 2, so
we take
$$|V|_{\geq 2k}=\bigoplus_{n\geq k}V_{2n} \mbox{ and }
|V|_{\geq 2k+1}=\bigoplus_{n\geq k}V_{2n+1}$$
for the parts of $V$ above a certain point, but in the same parity. 

Similarly, we move $N_{2k}$ into degree 0 by multiplication by $c^k$:
$$\Nb_{2k}:=c^kN_{2k}\subseteq |V|_0 \mbox{ and } \Nb_{2k+1}:=c^kN_{2k+1}\subseteq |V|_1.$$
Having established the framework, we will restrict the discussion to
the even part, leaving the reader to make the odd case explicit. 

 If $X$ is nonzero in even degrees, since
$X$   is small there is a highest degree $2a-2$ in which  $N$ is non-zero, and since
$N[1/c]=\Q[c, c^{-1}]\tensor V$ there is highest degree $2b$ in which
$N$ coincides with $|V|_0$. Accordingly, we have a finite filtration 
$$0=\Nb_{2a} \subseteq \Nb_{2a-2} \subseteq \cdots \subseteq \Nb_4\subseteq \Nb_2 \subseteq \cdots
\subseteq \Nb_{2b}=|V|_0. $$

We wish to consider two increasing filtrations on $|V|_0$
$$\cdots \subseteq |V|_{\geq 2k+2}\subseteq |V|_{\geq 2k}\subseteq |V|_{\geq 2k-2}\subseteq
\cdots \subseteq |V|_0$$
and 
$$\cdots \subseteq \Nb_{ 2k+2}\subseteq \Nb_{\geq 2k}\subseteq \Nb_{\geq 2k-2}\subseteq
\cdots \subseteq  |V|_0. $$

\subsection{Two conditions on wide spheres}
In crude terms, we will show the thick subcategory generated by $S^0$ consists of
objects so that (a) the dimensions of the spaces in the $V$- and $N$-filtrations agree and (b)
the $V$ filtration is slower than the $cN$ filtration. The purpose of
this subsection is to introduce the two conditions. 

\begin{cond}
\label{cond:thickS}
We say that a wide sphere is {\em untwisted} if it satisfies the
following two conditions
\begin{enumerate}
\item $\dim (\Nb_{i})=\dim (|V|_{\geq i})$ for all $i$
\item $V\cap cN =0$
\end{enumerate}
\end{cond}

We will be showing that these characterize the thick subcategory
generated by $S^0$. We must at least show that the conditions are
inherited by retracts, and this verification will lead us to some
useful introductory discussion.

\begin{lemma}
\label{lem:condretract}
Condition \ref{cond:thickS} is closed under passage to retracts. 
\end{lemma}

\begin{proof}
It is immediate that Conditon \ref{cond:thickS} (ii) is inherited by
retracts. We also note that Conditon \ref{cond:thickS} (ii) implies one of the
inequalities  for Conditon \ref{cond:thickS} (i):
$$\dim (\Nb_{i})\geq \dim (|V|_{\geq i}). $$

Now suppose $X=X'\oplus X''$ and that $X$ satisfies Condition
\ref{cond:thickS}. As observed already,  $X'$ and $X''$ both satisfy the second condition, and hence both 
satisfy the first condition with $=$ replaced by $\geq $. With lower
case letters denoting dimensions of vector spaces (for example
$n_a=\dim (N_a)$), this means we 
have a pair of  increasing sequences 
$$0=n'_a, n'_{a-2},n'_{a-4},  \cdots \mbox{ and }
0=v'_{\geq a}, v'_{\geq a-2},v'_{\geq a-4} \cdots $$
reaching $v'$
and a pair of increasing sequences
$$0=n''_a, n''_{a-2},n''_{a-4},  \cdots \mbox{ and }
0=v''_{\geq a}, v''_{\geq a-2},v''_{\geq a-4}\cdots $$
reaching $v''$. 
Since $X$ satisfies Condition \ref{cond:thickS}(i), the sum of the
first pair and the second pair give two equal sequences (i.e.,  the
sequence $n'_i+n''_i=n_i$ and the sequence $v'_{\geq i}+v''_{\geq
  i}=v_{\geq  i}$ are equal). Thus if one pair deviates from equality in the positive
direction, the other deviates in the negative direction. Since
Condition \ref{cond:thickS}(ii) shows there is no negative deviation,
we must have equality for both pairs throughout. 
\end{proof}

It is useful to be able to consider the changes of dimension and form
the  generating function. In fact to any wide sphere, we may associate to it two Laurent polynomials
\begin{itemize}
\item The geometric $T$-fixed point polynomial
$$p_T(t)=\sum_i\dim_{\Q} (V_i)t^i$$
\item The 1-Borel jump polynomal 
$$p_1(t)=\sum_i\dim_{\Q} (N_i/cN_{i+2})t^i$$
\end{itemize}
Condition \ref{cond:thickS}(i) is then equivalent to the condition 
$$p_T(t)=p_1(t). $$

\begin{remark}
We note that Condition \ref{cond:thickS}(i)  is not closed under passage to
retracts. Indeed, $S^z\vee S^{2-z}$ satisfies the first condition with 
$p_1(t)=p_T(t)=t^2+1. $
However  $S^z$ (with $p_T(t)=1$ and $p_1(T)=t^2$)
 and $S^{2-z}$ (with $p_T(t)=t^2$ and $p_1(T)=1$) do not.
  \end{remark}

\subsection{Attaching a $T$-fixed sphere}

To start with,   $S^0=(\Q [c]\lra \Q [c, c^{-1}]\tensor \Q)$ and then
direct sums of these model wedges of $T$-fixed spheres with  $N=\Q [c]\tensor V$, and of course
it is easy to see that Condition \ref{cond:thickS}  holds for these. 

However $N$ does not always sit 
so simply inside $\Q[c,c^{-1}]\tensor V$  for the objects built from
$S^0$. We may see this in a simple example. 

\begin{example}
Up to equivalence there are precisely three wide spheres with
$p_1(t)=p_T(t)=1+t^2$. Evidently in all cases $V=\Q\oplus \Sigma^2 \Q$, $\Nb_{2k}=0 $ for $k\geq 4$ and $\Nb_{2k}=|V|$ for $k\leq 0$. The
only question is how the 1-dimensional space $\Nb_2$ sits inside $|V|=
V_0\oplus V_2$. The three cases are $N_2=V_0$ (which is $S^z\vee
S^{2-z}$),  $N_2=V_2$ (which is $S^0\vee S^2$), and the third case
(giving just one isomorphism type)  in which $N_2$ is a 1-dimensional
subspace not equal to $V_0$ or $V_1$.

We note that the third example is the mapping cone $M_f$ for any
non-trivial map $f: S^1\lra S^0$ (in the semifree category,  there is
only one up to multiplication by a non-zero scalar). In this case up to
isomorphism,  $N_2$ is
generated by  $c^{-1}\tensor \iota_0+c^0\tensor \iota_2$. 

We observe then that the second and third of these three are in
$\thick (S^0)$, and we see that the first does not satisfy Condition
\ref{cond:thickS} (ii). 
\end{example}

\begin{lemma}
\label{lem:condaddcell}
Given a cofibre sequence, 
$$S^n \stackrel{f}\lra X\lra Y,  $$
if $X$ is a wide sphere then so is $Y$ and if $X$ in addition satisfies Condition \ref{cond:thickS} then so does $Y$.
\end{lemma}

\begin{proof}
Suppose first that $X$ is entirely in one parity. Without loss of
generality, we may suppose $X$ is in even degrees. 

If $n$ is odd then $\piA_*(S^n)$ is purely odd and we have a short
exact sequence
$$0\lra \piA_*(X)\lra \piA_*(Y)\lra  \piA_*(S^{n+1})\lra 0. $$
It follows that $Y$ is a wide sphere (i.e., the basing map is
injective). 
The condition on dimensions is immediate, since this is split as
vector spaces. For the second condition, we know that any element 
$(v,\lambda \iota ) \in V_Y\cap cN_Y$ with $v\in V_X$ must have $\lambda
\neq 0$ since $X$ satisfies the condition. However $\lambda \iota \not
\in c\Q [c]$. Altogether,   $Y$ satisfies
Condition \ref{cond:thickS}.  

Alternatively,  suppose $n$ is even. To calculate $[S^n, X]^T_*$ we
take an injective resolution of $X$. We argue that  this takes the form
$$ 0\lra X \lra e(V) \lra f(\Sigma^2V\tensor k[c]^{\vee})\lra 0. $$
To start with,  since $X$ is a wide sphere,  $X$
embeds in $e(V)$. The cokernel is zero at $T$ and hence of the form
$f(T)$ for some torsion $\Q [c]$-module $T$. At  $1$ the cokernel is $(\Q
[c,c^{-1}]\tensor V)/N$; since this is  divisible 
it is a sum of copies of $\Q[c]^{\vee}$. Finally, in view of Condition
\ref{cond:thickS} (i) $T=\Sigma^2 V\tensor \Q[c]^{\vee}$ as claimed.

Now we may use the Adams spectral sequence to see that $[S^n,
X]^T_0=V_n$. If the original  map $f$ is trivial, then 
 $\piA_*(Y)=\piA_*(X)\oplus \piA_*(S^{n+1})$, and the result is 
again clear. Otherwise we have a diagram
$$\diagram 
S^n \rto^f &X\\
\Sigma^n\Q[c]\rto^{\theta} \dto & N\dto \\
\Q[c,c^{-1}]\tensor \Sigma^n\Q\rto^{1\tensor \phi} & \Q [c,c^{-1}]\tensor V 
\enddiagram$$
This shows that since $\phi $ is mono then also $\theta$ is mono and
furthermore, by Condtion \ref{cond:thickS}(ii), $\theta$ is the
inclusion of a summand. It follows that the map $f$ is split. Indeed,
splittings of $\phi$ are given by codimension 1 free summands $N'$ of
$N$. This automatically has $\Nb'$ of codimension 1 in $|V|$. 
We need only choose $N'$  so that $\Nb'$ avoids $\theta
(\Sigma^n\Q)$. This gives a compatible splitting of $\phi$. 
It follows that in fact $Y$ is a retract of $X$, and the result
follows from Lemma \ref{lem:condretract}. 

Finally if $X$ has components in both even and odd degrees, then $X\simeq X_{ev}\vee X_{od}$ and we may argue as follows. Without
loss of generality we suppose $n$ is even. 
If $f$ maps purely into $X_{ev}$ or purely into $X_{od}$ the other factor is irrelevant and the above argument deals with this case. 
Otherwise $f$ has components mapping into both $X_{ev}$ and
$X_{od}$. The above argument shows that $\piA_*(Y)$ is a retract in
even degrees and it is unaltered in odd degrees.
\end{proof}

\subsection{Spectra built from $T$-fixed spheres}
We have now done the main work and can identify the thick subcategory
generated by $S^0$.

\begin{cor}
\label{cor:thickS}
The thick subcategory generated by $S^0$ consists of wide spheres
satisfying Condition \ref{cond:thickS}. 
\end{cor}

\begin{proof}
First, we observe that the thick subcategory $\thick (S^0)$ can be
constructed
by alternating the attachment of $T$-fixed spheres and taking
retracts; the fact that any element of $\thick (S^0)$ is a wide sphere
satisfying Condition \ref{cond:thickS} then follows from Lemmas
\ref{lem:condretract} and \ref{lem:condaddcell}. The point is that 
we must show that if  we construct $Z$ using a cofibre sequence $X\lra
Y \lra Z$ with $X,Y$ in the thick subcategory then $Z$ may be
constructed from $X$ by using the two processes. Formally, we are
applying induction on the number of cells, so we may suppose $Y$ is
constructed from the two processes. If $X$ is formed by attaching spheres, we may form $Z$ from
$Y$ by attaching the corresponding spheres. If $X$ is a retract of
$X'$ formed from spheres  then $f: X\lra Y$ extends over $X'=X\vee
X''$ by using $0$ on the second factor and  then $Z$ is a retract of $Z'$. 

Now  we show that any wide sphere satisfying Condition \ref{cond:thickS} is in the
thick subcategory generated by $S^0$. 
We argue by induction on the dimension of $|V|$. The result is obvious if $V=0$. 
Suppose that $X$ is a wide sphere satisfying the given condition and
that the result is proved when the geometric $T$-fixed points have lower dimension. 

Note that if $t^n$ is the smallest degree in which $p_T(t) $ is non-zero we may choose a vector $v\in V_n\setminus \Nb_{n+2}$. Accordingly
$X$  has a direct summand $\Q[c]\tensor v\lra \Q [c,c^{-1}]\tensor v$, which corresponds to a map 
$S^n\lra X$. Since $v\not \in \Nb_{n+2}$, the  quotient $Y$ again has injective basing map and obviously satisfies the polynomial condition. Since $n$ is the smallest
degree in which $V$ is non-zero, $v$,  the direct summand
$\Q \cdot v$ may be removed from $|V|$ without affecting the filtration condition. By induction $Y\in \thick (S^0)$, and hence $X\in \thick (S^0)$ as  required. 
\end{proof}

\subsection{Spectra built from representation spheres}
Since smashing with any sphere $S^{kz}$ is invertible, this allows us
to deduce the thick subcategory generated by any sphere. 
\begin{cor}
The thick subcatgory generated by $S^{kz}$ consists of 
wide spheres which are $k$-twisted in the sense that 
\begin{enumerate}
\item  $p_1(t)=t^{-2k}p_T(t)$.
\item $V\cap c^{k+1}N=0$
\end{enumerate}
\end{cor}

\begin{proof}
This is immediate from Corollary \ref{cor:thickS} and  the observations
$$V(S^{kz}\sm X)=V(X) \mbox{ and } N(S^{kz}\sm X)=c^{-k}N(X).$$
\end{proof}

\section{Beyond tori}
\label{sec:beyond}
It helps give a  proper perspective to describe the Balmer spectra of
other low rank groups. The answer follows easily from the work on the
circle group.

\subsection{Finite groups}
For a finite group,  the category of  rational $G$-spectra is equivalent to the product $\prod_{(H)}\mbox{$\Q [W_G(H)]$-mod}$, where the product is over
conjugacy classes of subgroups, 
$$\spc (\mbox{$G$-spectra})=\{ \wp_H \st (H) \in \sub(G)/G \}. $$
There are no containments between the subgroups and the  topological space is discrete. 

\subsection{The circle group $SO(2)$}
This is the special case of Theorem \ref{thm:allprimes} for the torus of rank 1, but we will use it below for $O(2)$ and $SO(3)$ so it is worth making it explicit. 

The closed subgroups of $SO(2)$ are the finite cyclic subgroups $C_n$ and $SO(2)$ itself. Each of the finite subgroups is cotoral 
in the whole group. Accordingly, by Theorem \ref{thm:allprimes} we have 
$$\spc (\cA (SO(2)))=\{\wp_{C_n}\st n\geq 1\}\cup \{ \wp_{SO(2)}\}
=\diagram
SO(2) &&&&\\
C_1\uto &C_2\ulto&C_3\ullto &\cdots &
\enddiagram$$
with  $\wp_{C_n} \subset \wp_{SO(2)}$. 

\subsection{The orthogonal group $O(2)$}
The homotopy category of $O(2)$-spectra splits into two parts. 
\begin{itemize}
\item The cyclic (or toral) part $\cC$ corresponding to the finite 
cyclic subgroups $C_n$ and the circle $T=SO(2)$. 
\item The dihedral part $\mcD$ corresponding to the finite 
dihedral subgroups isomorphic to $D_{2n}$ for some $n$ and the group $O(2)$ itself. 
\end{itemize}

The model of the cyclic part is $\cA (SO(2))[W]$ (where the group $W=O(2)/SO(2)$ acts on the ring $\cOcF$). Since all subgroups of $SO(2)$ are
fixed under conjugation by $W$,  
$$\spc (\cA (SO(2))[W]=\spc (\cA (SO(2)))=\{\wp_{C_n}\st n\geq 1\}\cup \{ \wp_{SO(2)}\}
=\diagram
SO(2) &&&&\\
C_1\uto &C_2\ulto&C_3\ullto &\cdots &
\enddiagram$$
with  $\wp_{C_n} \subset \wp_{SO(2)}$. 

The dihedral part corresponds to $W$-equivariant sheaves over the Polish Point $PP = \{0\} \cup \{ 1/2n \st n\geq 1\}$, where $0$ corresponds to the subgroup 
$O(2)$ and the point $1/2n$ corresponds to the conjugacy class of
$D_{2n}$. The action of $W$ on the fibre over $0$ is trivial. As a
topological space, the spectrum of this category 
is homeomorphic to $PP$. Indeed, one sees that whenever a finite complex has $O(2)$ in its geometric isotropy it has all but finitely many dihedral subgroups.

Altogether 
$$\spc (\mbox{$O(2)$-spectra/$\Q$}) = \{\wp_{C_n}\st n\geq 1\}\cup \{ \wp_{SO(2)}\} \amalg PP$$

\subsection{The rotation group $SO(3)$}
The homotopy category of $SO(3)$-spectra splits into three parts. 
\begin{itemize}
\item The cyclic (or toral) part $\cC$ corresponding to the finite 
cyclic subgroups isomorphic to $C_n$ for some $n$ and the maximal tori (each conjugate to $T= SO(2)$). 
\item The dihedral part $\mcD'$ corresponding to the finite 
dihedral subgroups isomorphic to $D_{2n}$ for some $n\geq 2$ and the normalizers of the maximal tori  (each conjugate to 
$O(2)$). The dihedral group of order 2 from $O(2)$ is conjugate in $SO(3)$ to the cyclic group of order 2, and so occurs under the 
first bullet point.  
\item Three conjugacy classess of exceptional subgroups conjugate to the tetrahedral, octahedral or icosohedral subgroups. 
\end{itemize}

The restriction along $i:O(2) \lra SO(3)$ is a  map $\rho :\mbox{$SO(3)$-spectra} \lra \mbox{$O(2)$-spectra} $
\begin{itemize}
\item $\rho^*$ gives a homeomorphism from the cyclic part of the
  spectrum of $SO(3)$-spectra to the cyclic part of the spectrum of
  $O(2)$-spectra (see Section \ref{sec:toral} for more
  details). 
$$\rho^*: \spc(\cC(SO(3)))\stackrel{\cong}\lra \spc(\cC (O(2))). $$
\item a homeomorphism on the space $\mcD'$ of dihedral groups of order $\geq 4$
$$\rho^*: \spc(\mcD (SO(3)))\stackrel{\cong}\lra \spc(\mcD' (O(2))). $$
\item $\rho^*(\wp_{D_2}^{O(2)})=\wp_{C_2}^{SO(3)}$. 
\end{itemize}

In other words
$$\spc (\cA (SO(3))=\spc (\cA(SO(2)))\amalg \mcD'(O(2))\amalg \{ \wp_{Tet}, \wp_{Oct}, \wp_{Icos}\} .$$

\section{Toral $G$-spectra}
\label{sec:toral}
Toral $G$-spectra are those whose geometric isotropy consists of
subgroups of a maximal torus. The point of this restriction is that
the analysis can essentially be reduced to the maximal torus. An abelian algebraic model and
calculation scheme is given in \cite{AGtoral},  and in \cite{qgtoralq}
it is shown to give a Quillen equivalence. 

We show here that the Balmer spectrum of finite toral $G$-spectra is
obtained by Going Up and Going Down from that for finite $T$-spectra
for the maximal torus $T$.

\subsection{The landscape}
The category of rational {\em toral} $G$-spectra is a retract of the
category of all rational $G$-spectra, namely the one obtained from by
applying the idempotent $e_{(T)}$ of the Burnside ring $A(G)=[S^0,S^0]^G$ corresponding to the 
subgroups of a maximal torus.

 An abelian algebraic model $\cA (G, toral)$ and calculation scheme was
given in \cite{AGtoral} and it is shown to be  Quillen equivalent to
rational toral $G$-spectra in \cite{qgtoralq}. Once a
monoidal Quillen equivalence has been established in the 
torus-equivariant case, one expects a monoidal equivalence in the
toral case to follow easily. We outline the conclusions here, but readers wanting
more details are referred to the source. 

The essential ingredients are the maximal torus $T$ in $G$, its
normalizer $N=N_G(T)$ and the action of the Weyl group $W=N/T$ on $T$.
One may consider the restriction functors
$$\diagram
\Gspectra \rto^{\res^G_N} \dto & \Nspectra \rto^{\res^N_T} \dto &
\Tspectra \dto \\
\cA (G,toral)\rto_{\theta_*} & \cA (N, toral)\rto & \cA (T)
\enddiagram$$
The functor labelled $\theta_*$  has the character of an extension of
scalars functor. The analysis proceeds by first understanding toral $N$-spectra, then
showing restriction $\res^G_N$ is faithful.

For toral $N$-spectra we proceed as follows. 
\begin{itemize}
\item A toral $N$-spectrum is essentially a $T$-spectrum with a
  homotopical $W$-action.  The restriction functor $\res^N_T$ on the
  toral part just forgets the $W$-action. On maps, 
$$[X,Y]^N=([\res^N_TX,\res^N_TY]^T)^W, $$
(i.e., the $W$-invariants of the $T$-maps), so the restriction is faithful but not full. 
\item Correspondingly, $\cA (N,toral)=\cA (T)[W]$ and restriction
  again just forgets the $W$-action. The $W$-action
  on $\cA (T)$ is suitably twisted. 
\end{itemize}

\subsection{Balmer spectra}
\label{subsec:spctoral}
We are now ready to describe the Balmer spectra of finite toral
$G$-spectra.

\begin{prop}
\label{prop:spcresNT}
The restriction functor $\res^N_T$ from toral $N$-spectra to
$T$-spectra induces a homeomorphism 
$$\spc ((toral-\Nspectra)^c)\cong \spc (\Tspectra^c)/W\cong \sub_a(T)/W. $$ 
The precisely similar statement holds for algebraic models. 
$$\spc (D(\cA(N,toral))^c)\cong \spc (D(\cA (T))^c)/W\cong \sub_a(T)/W. $$ 
\end{prop}

\begin{remark}
This is analagous to considering a ring $R$ with an action of a finite
group $W$ and then considering the ring map $u : R [W] \lra
R$. This induces a homeomorphism $\spc (R[W])\cong \spc (R)/W$. 
\end{remark}

We will return to the proof in Subsection \ref{subsec:toralproofs}
below.  

The restriction $\res^G_N$ from  toral $G$-spectra to toral $N$-spectra turns out to induce
a homeomorphism on Balmer spectra of small objects. 

\begin{prop}
\label{prop:spcresGN}
The restriction functor $\res^G_N$ from toral $G$-spectra to toral
$N$-spectra induces a homeomorphism 
$$\spc ((toral-\Gspectra)^c)\cong \spc ((toral-\Nspectra)^c)$$
The precisely similar statement holds for algebraic models. 
$$\spc (D(\cA(G,toral))^c)\cong \spc (D(\cA (N,toral))^c)$$
\end{prop}

\begin{remark}
This is analagous to considering a ring $R$ with an action of a finite
group $W$ and then considering the inclusion $\theta : R^W\lra
R[W]$. This induces a homeomorphism $\spec (R^W)=\spc (R^W)\cong \spc
(R[W])$. 
\end{remark}

We will return to the proof in Subsection \ref{subsec:toralproofs}
below.  

We note that $G$-conjugacy classes of subgroups of a maximal torus
correspond precisely to $W$-orbits of subgroups of $T$. Thus if we
write
$\sub_a(G,toral)$ for the closed subgroups of a maximal torus of $G$ we reach the
following conclusion.  
\begin{cor}
For a compact Lie group $G$ with maximal torus $T$ and $W=N_G(T)/T$, we have 
$$\spc ((toral-\Gspectra)^c)=\sub_a(T)/W =\sub_a(G, toral)/G $$
$$\spc (D(\cA (G,toral))^c)=\sub_a(T)/W =\sub_a(G,toral)/G. \qqed $$
\end{cor}

\subsection{Toral proofs}
\label{subsec:toralproofs}
\newcommand{\rhob}{\overline{\rho}}

We now prove the results described in Subsection \ref{subsec:spctoral}. 
One strategy would be to prove this by analogy with a ring $R$ with 
$W$-action and the sequence of ring maps 
$$R^W\lra R[W]\lra R  $$
corresponding to 
$$toral-\Gspectra \lra toral-\Nspectra\lra \Tspectra . $$
In this view the first map induces a homeomorphism of spectra and the
second is the quotient by $W$.  However we will instead treat the
general case of the map 
$$\rho: toral-\Gspectra \lra \Tspectra ,  $$
and show directly that it is the quotient by $W$.  The same argument
applies in particular when $G=N$. \\[2ex]

As for tori, we start by working to understand finitely generated
thick tensor ideals $\tthick(X)$.

\begin{prop}
If $X$ is a finite $G$-spectrum then 
$\tthick(X)$ only depends on $\gI (X)$. The geometric isotropy $\gI
(X)$ is a union of conjugacy classes of closed subgroups, it is closed
under passage to cotoral subgroups and it has finitely many maximal
elements $(K_1), \ldots , (K_r)$, where we suppose (without loss of
generality) that $K_i\subseteq T$. In this case
$$\tthick (X)=\tthick (\sigma^G_{(K_1)}, \ldots , \sigma^G_{(K_r)} ).$$
\end{prop}

\begin{proof}
If $K\in \gI (X)$ then from the analysis for tori, 
$$\sigma_K=\sigma_K^T\in \tthick (\res^G_T(X)).$$
This means there is a process for constructing $\sigma_K$ by use of
cofibre sequences, retracts and tensoring with finite $T$-spectra. We
will apply coinduction to this process. Some care is necessary for smashing with a
spectrum. 

\begin{lemma}
\label{lem:coind}
(a) If $X$ is a $G$-spetrum then 
$$F_T(G_+, \res^G_T(X)\sm Y)\simeq X\sm F_T(G_+, Y)$$

(b) Coinduction is monoidal on toral spectra: the unit map $S^0\lra
F_N(G_+, S^0)$ is an equivalence on toral parts, and for toral $N$-spectra
$A$ and $B$ the natural map  
$$F_N(G_+, A)\sm  F_N(G_+, B)\lra  F_N(G_+, A\sm B) $$
of toral $G$-spectra is an equivalence. 
\end{lemma}

\begin{proof}
Part (a) is the well known Frobenius reciprocity property. 

Part (b) follows from the statement that the unit 
$$X \lra F_N(G_+, \res^G_N X)$$
is an equivalence of toral spectra \cite[Subsection 10.A, or Corollary 7.11 for the
algebraic model]{AGtoral}.   At the very core, it is the statement
that for any compact Lie group $G/N$ has trivial rational homology. 
\end{proof}

\begin{cor} 
\label{cor:coind}
If $Z\in \tthick (A)$ then $F_T(G_+, Z)\in \tthick (F_T(G_+, A))$.
\end{cor}

\begin{proof}
Coinduction preserves triangles and retracts, so the only point is to
deal with smash products. 

We proceed in two steps. From $T$ to $N$ we use the fact that every
$T$-spectrum $A$ is the restriction of an $N$-spectrum $\tilde{A}$
(for example we may take $\tilde{A}$ to have `trivial $W$-action'
obtained by inducing up $A$ and then taking the idempotent summand
corresponding to the trivial representation). Hence we may apply Lemma \ref{lem:coind} (a)
to see
$$F_T(N_+, A\sm B)\simeq \tilde{A} \sm F_T(N_+, B). $$

Going from $N$ to $G$ we use the monoidal property of Lemma
\ref{lem:coind} (b). 
\end{proof}

Now, since $\sigma_K \in \tthick(\res^G_T X)$ we conclude from
Corollary  \ref{cor:coind} that 
$$F_T(G, \sigma_K)\in \tthick (F_T(G_+, \res^G_T(X)))=
\tthick (F_T(G_+,S^0) \sm X))\subseteq \tthick (X). $$

Now $\sigma^N_{(K)}$ is a retract of $F_T(N_+, \sigma_K)$, and hence
$\sigma^G_{(K)}$ is a retract of $F_T(G_+, \sigma_K)$ and hence in
$\tthick (X)$. 

By considering   $X$ as a $T$-spectrum we see, $\gI (X)$ has only finitely many
maximal conjugacy classes $(K_1), \ldots , (K_r)$.  Thus
$$\tthick(\sigma_{(K_1)}, \ldots \sigma_{(K_r)})\subseteq \tthick (X)
. $$
From the case of the torus we see
$$\res^G_T(X) \in \tthick (\res^G_T \sigma_{(K_1)},
\ldots,  \res^G_T \sigma_{(K_r)} ), $$
and applying coinduction 
$$X\in \tthick (\sigma_{(K_1)},
\ldots,   \sigma_{(K_r)} ). $$
\end{proof}

Finally, we turn to the statement about primes, continuing to write
$$\rho : \Gspectra \lra \Tspectra$$
for restriction. \\[2ex]

\begin{proof}
It is clear that the pullbacks of  $\wp_K$ and $\wp_{K'}$ are the same
if $K$ is $G$-conjugate to $K'$. Finally for injectivity, it is clear
that if $K$ is not conjugate to $K'$ then if $\dim(K)\geq \dim(K')$ then 
$$\sigma_{(K)}\in \rho^*(\wp_{K'})\setminus \rho^*(\wp_{K}). $$
Hence $\rho$ induces an injective  map 
$$\overline{\rho}: \sub_a(T)/W\lra \spc( (toral -\Gspectra)^c). $$

To see it is surjective, we repeat the proof of Theorem
\ref{thm:allprimes}, but with subgroups replaced by conjugacy
classes. 
\end{proof}

\end{document}